\documentclass[10pt]{article}
\usepackage{amsthm,amsfonts,amsmath,amssymb}
\usepackage{pxfonts}
\usepackage{euscript}
\usepackage{hyperref}
\usepackage{graphics}
\usepackage{epstopdf} 
\usepackage{xypic}
\usepackage[all,arrow,matrix]{xy}

\usepackage{graphicx}

\usepackage{sectsty}
\sectionfont{\Large}

\newcommand\void[1]       {}

\theoremstyle{definition}

\newtheorem{thm}{Theorem}[section]
\newtheorem{defn}[thm]{Definition}
\newtheorem{prop}[thm]{Proposition}
\newtheorem{cor}[thm]{Corollary}
\newtheorem{rem}[thm]{Remark}
\newtheorem{lem}[thm]{Lemma}

\newtheorem{expl}[thm]{Example}

\numberwithin{equation}{section}
\numberwithin{thm}{section}

\newcommand\nn             {\nonumber \\}

\newcommand\be            {\begin{equation}}
\newcommand\ee            {\end{equation}}
\newcommand\bea           {\begin{eqnarray}}
\newcommand\eea         {\end{eqnarray}}
\newcommand\bnu          {\begin{enumerate}}
\newcommand\enu          {\end{enumerate}}

\allowdisplaybreaks[1]

\newlength{\fighskip} \fighskip=2pt
\newlength{\figvskip} \figvskip=3pt

\newcommand*{\figbox}[2]{{
  \def\figscale{#1}
  \def\arraystretch{0.8}
  \arraycolsep=0pt
  \begin{array}{c}
    \vbox{\vskip\figscale\figvskip
      \hbox{\hskip\figscale\fighskip
        \includegraphics[scale=\figscale]{#2}}}
  \end{array}}}

\newcommand\id            {\mathrm{id}}
\newcommand\op          {\mathrm{op}}

\newcommand\mO          {\mathrm{O}}
\newcommand\bk          {\mathbf{H}}
\newcommand\stra          {\mathrm{stra}}
\newcommand\Disk          {\ED\mathrm{isk}}

\newcommand\rep     {\mathrm{Rep}}
\newcommand\fun     {\mathrm{Fun}}
\newcommand\Fun    {\EuScript{F}\mathrm{un}}
\newcommand\Ind  {\mathrm{Ind}}
\newcommand\Pres  {\mathrm{Pr}}

\newcommand\Or       {\mathrm{or}}

\newcommand\one    {\mathbf{1}}
\newcommand{\rev} {\mathrm{rev}}

\newcommand\cboxtimes {\boxtimes^\circlearrowright}

\newcommand{\pf}{\begin{proof}}
\newcommand{\epf}{\end{proof}}

\newcommand\Cb            {\mathbb{C}}

\newcommand\Rb            {\mathbb{R}}
\newcommand\Zb            {\mathbb{Z}}

\newcommand\EA           {\EuScript{A}}
\newcommand\EB           {\EuScript{B}}
\newcommand\EC           {\EuScript{C}}
\newcommand\ED           {\EuScript{D}}
\newcommand\EE          {\EuScript{E}}
\newcommand\EF          {\EuScript{F}}

\newcommand\EH         {\EuScript{H}}
\newcommand\EI           {\EuScript{I}}
\newcommand\EJ           {\EuScript{J}}
\newcommand\EK         {\EuScript{K}}
\newcommand\EL          {\EuScript{L}}
\newcommand\EM          {\EuScript{M}}
\newcommand\EN         {\EuScript{N}}

\newcommand\EP         {\EuScript{P}}
\newcommand\EQ         {\EuScript{Q}}
\newcommand\ER         {\EuScript{R}}

\newcommand\EV        {\EuScript{V}}

\newcommand\EX         {\EuScript{X}}

\begin{document}

\begin{center} \LARGE
Topological orders and factorization homology
\end{center}

\vskip 1em
\begin{center}
{\large
Yinghua Ai$^{a}$, Liang Kong$^{b,c}$,\,
Hao Zheng$^{d}$\,
~\footnote{Emails:
{\tt  yhai@math.tsinghua.edu.cn, kong.fan.liang@gmail.com, hzheng@math.pku.edu.cn}}}
\\[1em]
$^a$ Department of Mathematics, Tsinghua University \\
Beijing 100084, China
\\[0.5em]
$^b$ Department of Mathematics and Statistics\\
University of New Hampshire, Durham, NH 03824, USA
\\[0.5em]
$^c$ Center of Mathematical Sciences and Applications \\
Harvard University, Cambridge, MA 02138
\\[0.5em]
$^d$ Department of Mathematics, Peking University\\
Beijing 100871, China
\end{center}

\vskip 4em

\begin{abstract}
In the study of 2d (the space dimension) topological orders, it is well-known that bulk excitations are classified by unitary modular tensor categories. But these categories only describe the local observables on an open 2-disk in the long wave length limit. For example, the notion of braiding only makes sense locally. It is natural to ask how to obtain global observables on a closed surface. The answer is provided by the theory of factorization homology. We compute the factorization homology of a closed surface $\Sigma$ with the coefficient given by a unitary modular tensor category, and show that the result is given by a pair $(\bk, u_\Sigma)$, where $\bk$ is the category of finite-dimensional Hilbert spaces and $u_\Sigma\in \bk$ is a distinguished object that coincides precisely with the Hilbert space assigned to the surface $\Sigma$ in Reshetikhin-Turaev TQFT. We also generalize this result 
to a closed stratified surface decorated by anomaly-free topological defects of codimension 0,1,2. This amounts to compute the factorization homology of a stratified surface with a coefficient system satisfying an anomaly-free condition.
\end{abstract}

\tableofcontents

\vspace{1cm}

\section{Introduction}  \label{sec:introduction}

The study of topological orders has attracted a lot of attentions in recent years. In this work, we show how to compute the global observables of an anomaly-free 2d topological order on a closed 2d manifold. Here 2d is the space dimension. We also generalize the computation to closed stratified 2d manifolds decorated by anomaly-free topological defects. 

\medskip
By an anomaly-free 2d topological order, we mean a 2d topological order that can be realized by 2d lattice models \cite{kong-wen}. It is known that anomaly-free 2d topological orders are classified by unitary modular tensor categories (UMTC's) (up to $E_8$ quantum Hall states which will be ignored completely in this work). The objects in the UMTC correspond to topological excitations, which are all particle-like (i.e. 0d) topological defects (also called anyons) and completely local \cite{kitaev,lw-mod,kk}. These topological excitations can be moved (by string operators), fused and braided. These fusion-braiding structures are precisely given by the data and axioms of a UMTC. The trivial 2d topological order is given by the simplest UMTC: the category of finite-dimensional Hilbert spaces, denoted by $\bk$.

However, what has not been clarified nor emphasized in physics literature is that the fusion-braiding structures of topological excitations are only local observables defined in an open 2-disk. For example, the double braiding of two objects $x$ and $y$ in a UMTC, loosely speaking, corresponds to moving the particle-like topological excitation $x$ around another $y$ along a circular path. This double braiding only makes sense locally, for example, in an open 2-disk. If $x$ and $y$ are located on a sphere, a circular path around the topological excitation $y$ is contractible. Therefore, the double braiding does not make sense on a sphere at all. It means that the braiding structure is not a global observable. Then an obvious question is what the global observables are. The answer to this question is provided by the theory of factorization homology \cite{lurie,af}, which allows us to integrate local observables to obtain global observables.

\medskip
The theory of factorization homology originated from Beilinson-Drinfeld's theory of chiral homology \cite{bd} for conformal field theories. It was later generalized to topological field theories by Lurie \cite{lurie} and by Ayala, Francis, Tanaka \cite{af,aft-1,aft-2,afr} to include stratified spaces (with certain tangential structures such as framings). Some general machinery of computing factorization homology of surfaces with coefficient in generic braided tensor categories was developed by Ben-Zvi, Brochier and Jordan in \cite{bbj1,bbj2}. The factorization homology computed in this work does not depend on the framing (see Example\,\ref{expl:disk-alg}). 

Roughly speaking, observables in an open 2-disk form a so-called 2-disk algebra $A$ in a symmetric monoidal higher category (see Def.\,\ref{def:2-disk-alg}). For example, UMTC's are the examples of 2-disk algebras in the symmetric monoidal (2,1)-category of unitary categories 
(see for example \cite{wahl,sw}\cite[Example\,5.1.2.4]{lurie}\cite{fresse}). The theory of factorization homology says that such a 2-disk algebra can be integrated to give global observables on any (closed or open) surface $M$. These global observables, denoted by $\int_M A$, is called the factorization homology of $M$ with coefficient in $A$, and is defined as a colimit \cite{lurie,af,bbj1} (see Def.\,\ref{def:fh-1}). The space of local observables $A$ is referred to as the coefficient system of the factorization homology. Moreover, $\int_M A$ is a 0-disk algebra (see Rem.\,\ref{rem:fh-0disk}). For example, if $A$ is a UMTC $\EC$, the 0-disk algebra $\int_M A$ is given by a pair $(\EE, u)$, where $\EE$ is a unitary category and $u$ is a distinguished object in $\EE$.

Factorization homology satisfies a defining property called $\otimes$-excision property \cite{lurie,af} (see Thm.\,\ref{thm:tensor-excision}), which determines $\int_M A$ up to equivalence \cite[Thm.\,3.24]{af}. It turns out that this $\otimes$-excision property is a special case of a more general property: the pushforward property (see Sec.\,\ref{sec:def-fh}). It allows us to reduce the computation of the factorization homology of a surface to that of a lower dimensional manifold, and eventually to that of a 0-manifold.

Factorization homology can be also defined on stratified $n$-manifolds as certain colimits with a similar $\otimes$-excision property and pushforward property \cite{aft-1,aft-2}.

\medskip

The factorization homology of a surface $\Sigma$ with coefficient in a braided tensor category $\EC$ was first studied by Ben-Zvi, Brochier and Jordan \cite{bbj1}. Besides setting up the machinery of factorization homology in this case, they expressed $\int_\Sigma \EC$ beautifully as the category of modules over a $\EC$-algebra for a generic braided tensor category $\EC$ (see \cite[Thm.\,5.11]{bbj1}). In \cite{bbj2}, they generalized the computation to surfaces with boundaries and marked points. 

By working with arbitrary stratified surfaces but with much more restrictive coefficient systems, we obtain a very simple result (see Thm.\,\ref{thm:high-genus} and Thm.\,\ref{thm:higher-g-defects})

\begin{thm} \label{thm:main}
Let $\Sigma$ be a closed stratified surface with a coefficient system $A$ satisfying an anomaly-free condition (see Def.\,\ref{def:af-cond}). We have
$$
\int_\Sigma A \simeq (\bk,u_\Sigma), \quad\quad\quad \mbox{where $u_\Sigma$ is an object in $\bk$}.
$$
When $\Sigma$ has no 1-stratum and $A$ is determined by a single unitary modular tensor category $\EC$, $u_\Sigma$ is nothing but the Hilbert space that is assigned to the surface $\Sigma$ in Reshetikhin-Turaev 2+1D TQFT determined by $\EC$. 
\end{thm}

Physically, the only known global observable on $\Sigma$ for an anomaly-free topological order is the ground state degeneracy (GSD), which played a very important role in 80's in identifying fractional quantum Hall states as new types of phases beyond Landau's paradigm \cite{tw,ntw,nw}. Note that the Hilbert space $u_\Sigma$ in Thm.\,\ref{thm:main} gives exactly the GSD of the associated 2d topological order on $\Sigma$. Moreover, in Sec.\,\ref{sec:dr-fh}, we show that the dimensional reduction of 2d topological orders precisely coincides with the $\otimes$-excision property and the pushforward property of factorization homology. When a closed stratified surface is decorated by anomaly-free topological defects of codimension 0,1,2, the theory of factorization homology gives us a powerful tool to compute the GSD. We give some explicit computations in Sec.\,\ref{sec:fh-comp1} and \ref{sec:fh-comp2} and compare our results with results in physics literature (see for example \cite{hsw,hw,lww}). Importantly, factorization homology also provides a powerful tool to compute global observables on surfaces for anomalous topological orders. We will show that in the future.

\medskip
The layout of this paper is as follows: Sec.\,\ref{sec:umtc},\ref{sec:fh},\ref{sec:fh-comp} consists of pure mathematics while Sec.\,\ref{sec:to} interplays with physics.
In Sec.\,\ref{sec:umtc}, we recall the notion of unitary modular tensor category and a few useful results, and set our notations;
in Sec.\,\ref{sec:fh}, we briefly review the theory of factorization homology;
in Sec.\,\ref{sec:fh-comp}, we introduce the notion of an anomaly-free coefficient system and prove our main results;
in Sec.\,\ref{sec:to}, we explain the relation between factorization homology and topological orders, make some explicit computations and compare results with those in physics literature. 


\medskip
\noindent {\bf Acknowledgement}: We thank John Francis for teaching us a course on factorization homology. We thank David Ben-Zvi for pointing out the references \cite{bbj2,ft,schommer-pries} after the first version of our paper appeared on arXiv and helping us to understand the relation between these works and ours. In particular, Rem.\,\ref{rem:benzvi} is due to him. We also thank Daniel S. Freed for further clarifying this relation. We thank Yong-Shi Wu for clarifying the early history of the study of ground state degeneracy for topological orders. We also thank David Ben-Zvi, Adrien Brochier, John Francis and David Jordan for pointing out important references on 2-disk algebras. LK would like to thank Department of Mathematics at University of Tokyo and Institute of Advanced Study at Tsinghua University for their hospitality. YA is supported by NSFC under Grants No. 11001147 and No. 11471212. LK is supported by the Center of Mathematical Sciences and Applications at Harvard University. HZ is supported by NSFC under Grant No. 11131008.

\section{Elements of unitary modular tensor categories} \label{sec:umtc}

In this section, we recall some basic facts about unitary categories, unitary multi-fusion categories and unitary modular tensor categories, and set our notations.

\subsection{Unitary multi-fusion categories} \label{sec:umfc}



A {\it $\ast$-category} $\EC$ is a $\Cb$-linear category equipped with a functor $\ast: \EC \to \EC^\op$, which acts as the identity map on objects and is antilinear, involutive on morphisms, i.e. $\ast:\hom_\EC(x,y) \to \hom_\EC(y,x)$ is defined so that $(g \circ f)^\ast = f^\ast \circ g^\ast$, $(\lambda f)^\ast = \bar{\lambda} f^\ast$, $f^{\ast\ast} = f$ for $f\in \hom_\EC(x,y)$, $g\in \hom_\EC(y,z)$, $\lambda \in \Cb$. By a {\it functor} $F: \EC\to \ED$ between two $\ast$-categories $\EC$ and $\ED$, we mean a $\Cb$-linear functor such that $F(f^\ast)=F(f)^\ast$ for all morphisms $f$ in $\EC$. 

A $\ast$-category $\EC$ is called {\it unitary} if it has finitely many simple objects up to isomorphism, all hom spaces are finite-dimensional, and the $\ast$-operation is positive (i.e. $f\circ f^\ast =0$ implies $f=0$).  In this case, all hom spaces of $\EC$ are finite-dimensional Hilbert spaces. We denote the set of isomorphism classes of simple objects in such a category $\EC$ by $\mO(\EC)$. An example of unitary category is the category of finite-dimensional Hilbert spaces, denoted by $\bk$.  Other unitary categories are equivalent to finite direct sums of $\bk$.

\medskip
A {\it monoidal category} $\EC$ is a category equipped with a tensor product functor $\otimes: \EC \times \EC \to \EC$, a tensor unit $\one_\EC$ or just $\one$ for simplicity, associativity isomorphisms $\alpha_{a,b,c}: (a\otimes b) \otimes c \to a\otimes (b\otimes c)$ and unit isomorphisms $l_a: \one\otimes a \to a$,  $r_a: a\otimes \one \to a$ for $a,b,c\in \EC$, satisfying natural properties. We regard the opposite category $\EC^\op$ as a monoidal category equipped with the same tensor product $\otimes$. We use $\EC^\rev$ to denote the monoidal category which has the same underlying category $\EC$ but equipped with the reversed tensor product $a\otimes^\rev b := b\otimes a$.

A {\it unitary multi-fusion category} (UMFC) is a rigid unitary monoidal category $\EC$ such that $\alpha_{a,b,c}^\ast=\alpha_{a,b,c}^{-1}$, $l_a^\ast=l_a^{-1}$, $r_a^\ast=r_a^{-1}$. The rigidity of $\EC$ means that for every object $a$ in $\EC$ there is an object $a^\ast$ and duality maps $b_a: \one \to a^\ast\otimes a, \quad d_a: a\otimes a^\ast \to \one$, such that $(d_a\otimes \id_a)(\id_a\otimes b_a)=\id_a$ and $(\id_{a^\ast}\otimes d_a)(b_a\otimes \id_{a^\ast})=\id_{a^\ast}$. In this case, the functor $\delta: a\mapsto a^\ast$ gives a canonical monoidal equivalences between $\EC^\rev$ and $\EC^\op$. A UMFC is called a {\it unitary fusion category} if the tensor unit $\one$ is simple. A UMFC is called {\it indecomposable} if it is not a direct sum of two UMFC's.


\medskip
Let $\EC$ be a UMFC. A {\it left $\EC$-module} $\EM$ is a unitary category equipped with a left $\EC$-action functor $\odot: \EC \times \EM \to \EM$ that is unital and associative. By the rigidity of $\EC$, there is a natural isomorphism $\hom_\EM(a\odot x, y)\simeq\hom_\EM(x,a^\ast \odot y)$ for $a\in \EC$ and $x,y\in \EM$. A $\EC$-{\it module functor} $f: \EM \to \EM'$ between two left $\EC$-modules is a functor intertwining the $\EC$-actions. That is, there is a natural isomorphism $f(a\odot x)\simeq a\odot f(x)$ for $a\in \EC, x\in \EM$, satisfying some natural conditions (see for example \cite{ostrik}). We denote the category of $\EC$-module functors by $\fun_\EC(\EM,\EM')$; it is also a unitary category \cite{eno2002,ghr}. Moreover, $\fun_\EC(\EM,\EM)$ is a UMFC \cite{eno2002,ghr}. A {\it right $\EC$-module} $\EN$ is defined similarly as a left $\EC$-module. In this case, we have $\hom_\EN(x \odot b, y) \simeq \hom_\EN(x, y\odot b^\ast)$ for $b\in \EC$ and $x,y\in \EN$. If $\EM$ is a left (or right) $\EC$-module, the opposite category $\EM^\op$ is automatically a right (or left) $\EC$-module with the $\EC$-action $x\odot^\ast a := a^\ast \odot x$ (or $a\odot^\ast x:=x\odot a^\ast$) for $a\in \EC,x\in \EM$. Let $\EC$ and $\ED$ be two UMFC's. A {\it $\EC$-$\ED$-bimodule} $\EM$ is a left $\EC\boxtimes \ED^\rev$-module, where $\EC\boxtimes \ED^\rev$ is the Deligne tensor product \cite{deligne}. In this case, $\EM^\op$ is naturally a $\ED$-$\EC$-bimodule.

\medskip
For a left module $\EM$ over a UMFC $\EC$, the {\it inner hom} $[x,-]: \EM \to \EC$ for $x\in \EM$ is defined to be the right adjoint functor of $-\odot x: \EC \to \EM$, i.e. $\hom_\EM(a\odot x, y)\simeq \hom_\EC(a, [x,y]), \forall a\in \EC,y\in \EM$. It is known that $[x,y]$ exists for all $x,y\in \EM$ (\cite{ostrik,kong-zheng}).
Moreover, for $a,b\in \EC, x,y\in \EM$, we have
\be \label{eq:internal-hom}
a\otimes [x,y] \otimes b^\ast = [b\odot x, a\odot y].
\ee

A unitary category $\EM$ is automatically a left $\bk$-module and $\fun_\bk(\EM,\EM)$ coincides with the category of functors from $\EM$ to $\EM$. 
For a  UMFC $\EC$, a left $\EC$-module structure on $\EM$ is equivalent to a monoidal functor $\EC \to \fun_\bk(\EM,\EM)$.
\begin{defn} \label{def:closed-module}
A left $\EC$-module $\EM$ is called {\it closed} if the canonical monoidal functor $\EC \to \fun_\bk(\EM, \EM)$ is a monoidal equivalence.
\end{defn}

\subsection{Tensor products of module categories}
Let $\EC,\ED,\EE,\EF$ be UMFC's throughout this subsection.

\medskip
For a right $\EC$-module $\EM$, a left $\EC$-module $\EN$ and a unitary category $\EP$, a {\it balanced $\EC$-module functor} $F: \EM \times \EN \to \EP$ is a functor equipped with isomorphisms
\be \label{eq:balance}
F(x\odot a, y) \simeq F(x, a\odot y), \quad\quad
\mbox{for $x\in \EM, a\in \EC,y\in \EN$}
\ee
that are natural in all three variables and satisfy some natural properties (see \cite{tam,eno2009}). The {\it tensor product} of $\EM$ and $\EN$ over $\EC$ is a unitary category $\EM\boxtimes_\EC\EN$, together with a balanced $\EC$-module functor $\boxtimes_\EC:\EM\times\EN\to\EM\boxtimes_\EC\EN$, such that, for every unitary category $\EP$ and a balanced $\EC$-module functor $F: \EM\times \EN \to \EP$, there is a unique (up to isomorphism) functor $\tilde{F}$ such that the following diagram
$$
\xymatrix{
\EM \times \EN \ar[r]^{\boxtimes_\EC}  \ar[rd]_F & \EM \boxtimes_\EC \EN  \ar[d]^{ \tilde{F} }  \\
& \EP
}
$$
is commutative up to isomorphism (see for example \cite{tam,eno2009}). Such a tensor product always exists. Indeed, there is a canonical equivalence \cite[Cor.\,2.2.5]{kong-zheng}:
\be \label{eq:xRy}
\EM^\op\boxtimes_\EC\EN \xrightarrow{\simeq} \fun_\EC(\EM,\EN) \quad\quad
\text{defined by} \quad\quad x\boxtimes_\EC y\mapsto 
[-,x]^\ast\odot y,
\ee
where $x\boxtimes_\EC y$ is the image of $(x,y)$ under the functor $\boxtimes_\EC$. When $\EC=\bk$, $\EM\boxtimes_\bk\EN$ is just the Deligne tensor product $\EM\boxtimes \EN$.

\medskip

If $\EM$ is a $\EC$-$\ED$-bimodule and $\EN$ is a $\ED$-$\EE$-bimodule, then $\EM\boxtimes_\ED \EN$ is a $\EC$-$\EE$-bimodule. If $\EP$ is a $\EE$-$\EF$-bimodule, we have the associativity equivalence $(\EM\boxtimes_\ED\EN)\boxtimes_\EE\EP \simeq \EM \boxtimes_\ED (\EN\boxtimes_\EE \EP)$ defined by $(m\boxtimes_\ED n) \boxtimes_\EE p \mapsto m\boxtimes_\ED (n\boxtimes_\EE p)$ (see for example \cite{tam,eno2009}).
Therefore, we simply denote both categories by $\EM\boxtimes_\ED \EN\boxtimes_\EE \EP$. Moreover, we also have canonical equivalences $\EC\boxtimes_\EC \EM \simeq \EM \simeq \EM\boxtimes_\ED \ED$ defined by $c\boxtimes_\EC m \mapsto c\odot m \mapsto (c\odot m)\boxtimes_\ED \one_\ED$ for $c\in \EC, m\in \EM$.

\begin{rem} \label{rem:tensor-bimod}
If $\EM$ is a $\EC$-$\ED$-bimodule and $\EN$ is a $\EC$-$\EE$-bimodule, then the canonical equivalence $\EM^\op \boxtimes_\EC \EN \xrightarrow{\simeq} \fun_\EC(\EM, \EN)$ defined by Eq.\,(\ref{eq:xRy}), is an equivalence between $\ED$-$\EE$-bimodules. Indeed, we have
$(d\odot^\ast x)\boxtimes_\EC (y\odot e) = (x\odot d^\ast) \boxtimes_\EC (y\odot e) \mapsto [-, x\odot d^\ast]^\ast \odot (y\odot e) \simeq ([-\odot d, x]^\ast \odot y)\odot e$ for $d\in\ED$, $e\in\EE$, $x\in\EM$, $y\in\EN$.
\end{rem}

Let $\EM$ be a $\EC$-$\ED$-bimodule and $\EN$ a $\ED$-$\EC$-bimodule. We have a tensor product $\EM\boxtimes_{\EC^\rev\boxtimes \ED} \EN$. In order to emphasize its cyclic nature, we would also like to introduce another convenient notation:
$$
\cboxtimes_\EC (\EM\boxtimes_\ED \EN):=\EM\boxtimes_{\EC^\rev \otimes \ED} \EN.
$$
This cyclic tensor product $\cboxtimes_\EC$ was carefully studied under the name of categorical-valued trace in \cite{fss}. Let $\EC_0, \cdots, \EC_n$ be UMFC's and $\EC_0=\EC_n$, and let $\EM_i$ be $\EC_{i-1}$-$\EC_i$-bimodule for $i=1, \cdots, n$. We would like to introduce a symmetric notation:
$$
\cboxtimes_{(\EC_0, \cdots, \EC_{n-1})} (\EM_1, \cdots, \EM_n) := \cboxtimes_{\EC_0} (\EM_1\boxtimes_{\EC_1} (\EM_2 \boxtimes_{\EC_2} \cdots \boxtimes_{\EC_{n-1}} \EM_n)).
$$


\subsection{Unitary modular tensor categories}

A {\it braided monoidal category} $\EC$ is a monoidal category equipped with a braiding $a\otimes b \xrightarrow{c_{a,b}} b\otimes a$ for all $a,b\in \EC$. We use $\overline{\EC}$ to denote the same monoidal category but equipped with a new braiding structure given by the anti-braidings in $\EC$.
The {\it M\"{u}ger center} of $\EC$, denoted by $\EC'$, is the full subcategory of $\EC$ consisting of those objects $x\in \EC$ such that $c_{y,x}\circ c_{x,y}=\id_{x\otimes y}$ for all $y\in \EC$. 

A {\it unitary braided fusion category} $\EC$ is a unitary fusion category with a braiding structure such that $c_{a,b}^\ast=c_{a,b}^{-1}$ for $a,b\in \EC$. A unitary braided fusion category has a canonical pivotal structure, which is equivalent to the structure of a twist $\theta_a: a \xrightarrow{\simeq} a, \forall a\in \EC$ such that $\theta_\one=\id_\one$ and $\theta_{a\otimes b} = c_{b,a}\circ c_{a,b} \circ (\theta_a \otimes \theta_b)$. Moreover, this pivotal structure is spherical, i.e. $\theta_{a^\ast}=(\theta_a)^\ast$ for all $a\in \EC$ \cite{kitaev}. A {\it unitary modular tensor category} (UMTC) is a unitary braided fusion category equipped with the canonical spherical structure such that it satisfies the non-degeneracy condition: $\EC' \simeq \bk$ \cite{rt,turaev}.

\medskip
Given a monoidal category $\EC$, there is a canonical braided monoidal category $Z(\EC)$ associated to $\EC$, called the {\it Drinfeld center} of $\EC$. Its objects are pairs $(z, c_{z,-})$, where $c_{z,-}: z\otimes - \to -\otimes z$ (called a {\it half-braiding}) is a natural isomorphism satisfying some natural conditions. It is clear that $Z(\EC^\rev)$ can be identified with $\overline{Z(\EC)}$ as braided monoidal categories. If $\EC$ is a UMFC, the Drinfeld center $Z(\EC)$ is a UMTC \cite{mueger2}. When $\EC$ is a UMTC, we have $Z(\EC) \simeq \EC \boxtimes \overline{\EC}$ as UMTC's \cite{mueger2}. 

\medskip

We recall the notions of a multi-fusion bimodule introduced in \cite{kong-zheng}. Similar notions in much more general settings were introduced earlier in \cite{lurie,francis,ginot,bbj2}.

\begin{defn}  \label{def:monoidal-module}
Let $\EC,\ED$ be UMTC's. A {\em multi-fusion $\EC$-$\ED$-bimodule} is a unitary multi-fusion category $\EK$ equipped with a braided monoidal functor $\phi_\EK:\overline{\EC}\boxtimes \ED\to Z(\EK)$.
A multi-fusion $\EC$-$\ED$-bimodule $\EK$ is said to be {\it closed} if $\phi_\EK$ is an equivalence. If a closed multi-fusion $\EC$-$\ED$-bimodule exists, then the UMTC's $\EC$ and $\ED$ are said to be {\it Witt equivalent} (\cite{dmno}).
We say that two multi-fusion $\EC$-$\ED$-bimodules $\EK$ and $\EL$ are {\it equivalent} if there is a monoidal equivalence $\EK\simeq\EL$ such that the composition of $\phi_\EK$ with the induced equivalence $Z(\EK)\simeq Z(\EL)$ is isomorphic to $\phi_\EL$.
\end{defn}

\void{

We recall the notions of multi-fusion left/right/bi-modules introduced in \cite{kong-zheng}. Similar notions in much more general settings were introduced earlier in \cite{lurie,francis,ginot,bbj2}.
\begin{defn}  \label{def:monoidal-module}
Let $\EC,\ED$ be UMTC's.
\bnu
\item A {\em multi-fusion left $\EC$-module} $\EM$ is a unitary multi-fusion category $\EM$ equipped with a braided monoidal functor $\phi_\EM:\EC\to \overline{Z(\EM)}$;
\item A {\em multi-fusion right $\ED$-module} $\EN$ is a unitary multi-fusion category $\EN$ equipped with a braided monoidal functor $\phi_\EN:\ED\to Z(\EN)$.
\item A {\em multi-fusion $\EC$-$\ED$-bimodule} $\EK$ is a unitary multi-fusion category $\EK$ equipped with a braided monoidal functor $\phi_\EK:\overline{\EC}\boxtimes \ED\to Z(\EK)$.
\enu
A multi-fusion $\EC$-$\ED$-bimodule is said to be {\it closed} if $\phi_\EK$ is an equivalence. If a closed multi-fusion $\EC$-$\ED$-bimodule exists, then the UMTC's $\EC$ and $\ED$ are said to be {\it Witt equivalent} (\cite{dmno}).
\end{defn}

\begin{rem}
If $\EK$ is a multi-fusion $\EC$-$\ED$-bimodule, then $\EK^\rev$ is clearly a multi-fusion $\ED$-$\EC$-bimodule. Moreover, if $\EK$ is closed, so is $\EK^\rev$.
\end{rem}

Let $\EM$ be a right multi-fusion $\ED$-module. The braided monoidal functor $\phi_\EM: \ED \to Z(\EM)$ can be replaced by a monoidal functor $f_\EM: \ED \to \EM$ with a half-braiding $c_{a,-}: f_\EM(a)\otimes - \xrightarrow{} -\otimes f_\EM(a)$ for each $a\in \ED$ such that $c_{a\otimes b,y}=(c_{a,y}\otimes \id_b)(\id_a\otimes c_{b,y}), \forall y\in \EM$ and the following diagram
$$
\xymatrix{
f_\EM(a) \otimes f_\EM(b) \ar[r]^\sim \ar[d]_{c_{a,f_\EM(b)}} & f_\EM(a\otimes b) \ar[d]^{f_\EM(c_{a,b})} \\
f_\EM(b) \otimes f_\EM(a) \ar[r]^\sim & f_\EM(b\otimes a)\,
}
$$
is commutative. Similarly, for a multi-fusion $\EC$-$\ED$-bimodule $\EM$, the defining data $\phi_\EM: \overline{\EC} \boxtimes \ED \to Z(\EM)$ can be replaced by two monoidal functors $\EC \xrightarrow{L} \EM \xleftarrow{R} \ED$ satisfying natural conditions.

\begin{defn} \label{def:monoidal-module-map}
Let $\EC$ be a UMFC and $\EM,\EN$ right multi-fusion $\EC$-modules. A {\em monoidal $\EC$-module functor $F:\EM\to\EN$} is a monoidal functor equipped with an isomorphism of monoidal functors $F\circ f_\EM\simeq f_\EN: \EC\to\EN$ such that the evident diagram
\be  \label{diag:m-m-map}
\xymatrix{
  F(f_\EM(a)\otimes x) \ar[r]^\sim \ar[d]_\sim & F(x\otimes f_\EM(a)) \ar[d]^\sim \\
  f_\EN(a)\otimes F(x) \ar[r]^\sim & F(x)\otimes f_\EN(a) \\
}
\ee
is commutative for $a\in\EC$ and $x\in\EM$. Two right multi-fusion $\EC$-modules are said to be equivalent if there is an invertible monoidal $\EC$-module functor $F:\EM\to\EN$.
\end{defn}

}

Let $\EC,\ED,\EE$ be UMTC's. Let $\EM$ be a multi-fusion $\EC$-$\ED$-bimodule and $\EN$ a multi-fusion $\ED$-$\EE$-bimodule. The category $\EM\boxtimes_\ED \EN$ has a natural structure of a UMFC with the monoidal structure defined by $(a\boxtimes_\ED b) \otimes (c\boxtimes_\ED d) \simeq (a\otimes c)\boxtimes_\ED(b\otimes d)$. Moreover, $\EM\boxtimes_\ED\EN$ satisfies the usual universal property. More precisely, if $\EK$ is another UMFC and there is a monoidal functor $F: \EM\boxtimes \EN \to \EK$ such that $F$ is also a balanced $\ED$-module functor, then there is a unique monoidal functor $\tilde{F}: \EM\boxtimes_\ED \EN \to \EK$ up to isomorphism such that $F\simeq\tilde{F} \circ \boxtimes_\ED$.

\begin{thm}[\cite{kong-zheng} Thm.\,3.3.6] \label{thm:closed}
Let $\EC,\ED,\EE$ be UMTC's. If $\EM$ is a closed multi-fusion $\EC$-$\ED$-bimodule, and $\EN$ is a closed multi-fusion $\ED$-$\EE$-bimodule, then $\EM\boxtimes_\ED \EN$ is a closed multi-fusion $\EC$-$\EE$-bimodule.
\end{thm}

\void{
\begin{thm}[\cite{kong-zheng} Cor.\,2.5.3]
If $\EE$ is a closed multi-fusion $\bk$-$\bk$-bimodule, then there is a $\bk$-module $\EM$ (i.e. a unitary category) such that $\EE \simeq \fun_\bk(\EM,\EM)$.
\end{thm}
}

We recall the main result in \cite{kong-zheng} that can be generalized to the unitary case automatically. Let $\mathbf{UMFC}^{\mathrm{ind}}$ be the category of indecomposable UMFC's with morphisms given by the equivalence classes of non-zero bimodules. Let $\mathbf{UMTC}$ be the category of UMTC's with morphisms given by the equivalence classes of closed multi-fusion bimodules.

\begin{thm}[\cite{kong-zheng} Thm.\,3.3.7] \label{thm:kz}
There is a well-defined functor $Z: \mathbf{UMFC}^{\mathrm{ind}} \to \mathbf{UMTC}$ given by $\EC \mapsto Z(\EC)$ on object and ${}_\EC \EM_\ED \mapsto Z(\EM):=\fun_{\EC|\ED}(\EM,\EM)$ on morphism. Moreover, the functor $Z$ is fully faithful.
\end{thm}

More explicitly, Thm.\,\ref{thm:kz} implies the following result. Let $\EC,\ED,\EE$ be UMFC's. Let $\EM$ be a multi-fusion $\EC$-$\ED$-bimodule and $\EN$ a multi-fusion $\ED$-$\EE$-bimodule. The assignment $f\boxtimes_{Z(\ED)}g \mapsto f\boxtimes_\ED g$ defines an equivalence between two multi-fusion $Z(\EC)$-$Z(\EE)$-bimodules:
\be \label{eq:MMNN}
\fun_{\EC|\ED}(\EM, \EM) \boxtimes_{Z(\ED)} \fun_{\ED|\EE}(\EN,\EN)
\simeq \fun_{\EC|\EE}(\EM\boxtimes_\ED \EN, \EM\boxtimes_\ED \EN).
\ee

The following two corollaries are useful to us.
\begin{cor} \label{cor:tensor-fun}
Let $\EC,\ED$ be UMFC's. Given a braided monoidal equivalence $Z(\ED)\simeq Z(\EC)$, there is a unique (up to equivalence) left $\EC$-module $\EM$ such that $\ED^\rev \simeq \fun_\EC(\EM,\EM)$ as multi-fusion $Z(\EC)$-$\bk$-bimodules. Moreover, there is a canonical monoidal equivalence
\be \label{eq:CC=FCC}
\EC\boxtimes_{Z(\EC)} \ED^\rev \simeq \fun_\bk(\EM,\EM) 
\quad \quad \text{define by} \quad \quad c\boxtimes_{Z(\EC)}d\mapsto c\odot - \odot d.
\ee
\end{cor}

\begin{cor} \label{cor:0-cell-condition}
Let $\EC_0, \cdots, \EC_n$ be UMTC's and $\EC_n=\EC_0$. Let $\EM_i$ be a closed multi-fusion
$\EC_{i-1}$-$\EC_i$-bimodule, $i=1,\cdots, n$. There is a unique (up to equivalence) unitary category $\EP$ such that
$$
\cboxtimes_{(\EC_0,\cdots,\EC_{n-1})}(\EM_1,\cdots,\EM_n) \simeq \fun_\bk(\EP,\EP).
$$
\end{cor}

\section{Factorization homology} \label{sec:fh}

In this section, we review 
the definition and fundamental properties of factorization homology (of stratified spaces) and show some examples for later use.

\subsection{Factorization homology} \label{sec:def-fh}
\begin{defn}
We define $\text{Mfld}_n^\Or$ to be the topological category whose objects are oriented $n$-manifolds without boundary. For any two oriented $n$-manifolds $M$ and $N$, the morphism space $\text{Hom}_{\text{Mfld}_n^\Or}(M,N)$ is the space of all orientation-preserving embeddings $e:M\to N$, endowed with the compact-open topology. We define $\EM\text{fld}_n^\Or$ to be the symmetric monoidal $\infty$-category associated to the topological category $\text{Mfld}^\Or_n$ \cite{lurie0}. The symmetric monoidal structure is given by disjoint union.
\end{defn}

\begin{defn}
The symmetric monoidal $\infty$-category $\Disk_n^\Or$ is the full subcategory of $\EM\text{fld}_n^\Or$ whose objects are disjoint union of finitely many $n$-dimensional Euclidean spaces $\coprod_{I}\Rb^n$ equipped with the standard orientation.
\end{defn}


\begin{defn} \label{def:2-disk-alg}
Let $\EV$ be a symmetric monoidal $\infty$-category. An {\em $n$-disk algebra} in $\EV$ is a symmetric monoidal functor $A:\Disk_n^\Or\to\EV$.
\end{defn}

\begin{rem}
An $n$-disk algebra in this work is called an oriented $n$-disk algebra in \cite{af}. Since we only consider oriented $n$-disk algebras in this work, we drop the ``oriented" for simplicity.
\end{rem}

\begin{expl} \label{expl:disk-alg}
In this work, we are mainly interested in examples of 0-,1-,2-disk algebras in the symmetric monoidal $(2,1)$-category of unitary categories, denoted by $\EV_{\mathrm{uty}}$. The tensor product of $\EV_{\mathrm{uty}}$ is given by Deligne tensor product $\boxtimes$.
\bnu
\item A 2-disk algebra in $\EV_{\mathrm{uty}}$ is a unitary braided monoidal category (an $E_2$-algebra) together with a twist: $\theta_a: a \xrightarrow{\simeq} a, \forall a\in \EC$ such that $\theta_\one=\id_\one$ and $\theta_{a\otimes b} = c_{b,a}\circ c_{a,b} \circ (\theta_a \otimes \theta_b)$ (see for example \cite{wahl,sw}\cite[Example\,5.1.2.4]{lurie}\cite{fresse}). The twist is needed for the factorization homology to be defined on surfaces without framing. A unitary braided fusion category is automatically equipped with a canonical spherical structure, which automatically includes the structure of a twist. Therefore, a unitary braided fusion category gives a 2-disk algebra in $\EV_{\mathrm{uty}}$. In this work, we are only interested in such 2-disk algebras in $\EV_{\mathrm{uty}}$ satisfying an additional anomaly-free condition: non-degenerate unitary braided fusion categories, or equivalently, unitary modular tensor categories (UMTC's).

\item A 1-disk algebra (or an $E_1$-algebra) in $\EV_{\mathrm{uty}}$ is a unitary monoidal category. In this work, we are only interested in a special class of such 1-disk algebras in $\EV_{\mathrm{uty}}$: unitary multi-fusion categories (UMFC's).

\item A 0-disk algebra (or an $E_0$-algebra) in $\EV_{\mathrm{uty}}$ is a pair $(\EP,p)$, where $\EP$ is a unitary category and $p\in \EP$ is a distinguished object.
\enu
\end{expl}

We need the notion of colimit in the $\infty$-categorical context, see the reference \cite{lurie0}.
Let $F:\EI\to\EV$ be a functor of $\infty$-categories,
we denote its colimit by $\text{Colim}(\EI\xrightarrow{F}\EV)$. 

\begin{defn} \label{def:fh-1}
Let $\EV$ be a symmetric monoidal $\infty$-category, $M$ be an oriented $n$-manifold, and let $A:\Disk^\Or_n\to\EV$ be an $n$-disk algebra in $\EV$. The {\it factorization homology} of $M$
with coefficient in $A$ is an object of $\EV$ given by the following expression
$$\int_MA:=\text{Colim}\left((\Disk_n^\Or)_{/M}\to\Disk_n^\Or\xrightarrow{A}\EV\right),$$
where $(\Disk_n^\Or)_{/M}$ is the over category of $n$-disks embedded in $M$.
\end{defn}

\begin{rem} \label{rem:fh-0disk}
Let $\one_n$ be the trivial $n$-disk algebra, which assigns to each $\Rb^n$ the unit object $\one_\EV$ of $\EV$. It is clear that $\int_M\one_n \simeq \one_\EV$. The canonical morphism $\one_n \to A$ then induces a morphism $\one_\EV \to \int_MA$. This is to say, the factorization homology $\int_MA$ is not merely an object of $\EV$, but also equipped with the structure of a 0-disk algebra. This additional structure is very important in applications to topological orders, because it corresponds to the notion of GSD. For this purpose, we will compute factorization homology explicitly as a 0-disk algebra.
\end{rem}

Factorization homology generalizes straightforwardly to manifolds with boundaries. Precisely, we enlarge $\EM\text{fld}^\Or_n$ to the symmetric monoidal $\infty$-category $\EM\text{fld}^{\partial,\Or}_n$ of $n$-manifolds possibly with boundaries and orientation-, boundary-preserving embeddings; enlarge $\Disk^\Or_n$ to the full subcategory $\Disk^{\partial,\Or}_n$ consisting of disjoint unions of $\Rb^n$'s and $\Rb^{n-1}\times [0,1)$'s. Then define factorization homology in the same way.

\begin{defn} \label{def:collar}
A {\it collar-gluing} among $n$-manifolds $M$ is a continuous map $f: M \to [-1,1]$ to the closed interval such that the restriction of $f$ to the preimage of $(-1,1)$ is a manifold bundle. We denote a collar-gluing $f: M \to [-1,1]$ simply by the open cover $M_{[-1,1)}\cup_{M_{\{0\}}\times \Rb} M_{(-1,1]}\cong M$, where $M_{[-1,1)}$, $M_{(-1,1]}$ and $M_{\{0\}}$ are the preimages of $[-1,1)$, $(-1,1]$ and $\{ 0\}$.
\end{defn}

Note that $\int_{M_{\{0\}}\times \Rb} A$ has a canonical 1-disk algebra structure. Moreover,
$\int_{M_{[-1,1)}} A$ has a structure of the right module over $\int_{M_{\{0\}}\times \Rb} A$. Similarly, $\int_{M_{(-1,1]}} A$ has a structure of the left module over $\int_{M_{\{0\}}\times \Rb} A$. Therefore, one can talk about their tensor product over $\int_{M_{\{0\}}\times \Rb} A$ (if it exists). 


\begin{thm} [\cite{af} Lem.\,3.18] \label{thm:tensor-excision}
Suppose $\EV$ is presentable and the tensor product $\otimes: \EV \times \EV \to \EV$ preserves small colimits for both variables. 
Then the factorization homology satisfies the $\otimes$-excision property. That is, for any collar-gluing
$M_{[-1,1)}\cup_{M_{\{0\}}\times \Rb} M_{(-1,1]}\cong M$, we have a canonical equivalence:
\be \label{eq:excision}
\int_M A \simeq \left(\int_{M_{[-1,1)}} A\right) \otimes_{\int_{M_{\{0\}} \times \Rb} A} \left( \int_{M_{(-1,1]}} A \right).
\ee
\end{thm}

\begin{rem}
The $\otimes$-excision property uniquely determines $\int_M A$ up to equivalence \cite[Thm.\,3.24]{af}. Therefore, 
we can also take this $\otimes$-excision property as our working definition of factorization homology. 
\end{rem}

\begin{rem}
The symmetric monoidal $(2,1)$-category of unitary categories $\EV_{\mathrm{uty}}$ does not satisfy the condition of Thm.\,\ref{thm:tensor-excision}. However, the symmetric monoidal $(2,1)$-category of presentable $\Cb$-linear categories $\Pres_\Cb$ satisfies that condition and there is a fully faithful embedding $\Ind: \EV_{\mathrm{uty}} \hookrightarrow \Pres_\Cb$ that carries a unitary category to its Ind-completion. Actually, a unitary category is nothing but a finite direct sum of $\bk$ and its Ind-completion is nothing but a direct sum of the category of vector spaces. Moreover, it is clear that this embedding preserves  tensor products over UMFC's. Therefore, one may use the $\otimes$-excision property to compute factorization homology in $\EV_{\mathrm{uty}}$ by keeping in mind that $\EV_{\mathrm{uty}}$ is embedded in $\Pres_\Cb$. See \cite{bbj1} for a treatment of a much more sophisticated case.
\end{rem}

It turns out that this $\otimes$-excision property is a special case of a more general property: the {\it pushforward property}, which is useful in the study of topological orders. Let $M$ be an oriented $m$-manifold, $N$ an oriented $n$-manifold, possibly with boundary, and $f: M \to N$ a map which fibers over the interior and boundary of $N$. In \cite{af}, they defined a functor $f_{\ast}A:\Disk^{\partial,\Or}_n\to\EV$ as follows. Given an embedding $e:U\to N$ where $U=\Rb^n$ or $\Rb^{n-1}\times [0,1)$, define
$$
(f_{\ast}A)(U):=\int_{f^{-1}(e(U))}A.$$
In other words, by integrating along the fiber, we get a $\Disk^{\partial,\Or}_n$ algebra $f_{\ast}A$.
By \cite[Prop.\,3.23]{af}, if $\EV$ satisfies the condition of Thm.\,\ref{thm:tensor-excision}, there is a canonical equivalence
\be \label{eq:pushforward}
\int_N f_\ast A \xrightarrow{\simeq}  \int_M A.
\ee
We show in Sec.\,\ref{sec:dr-fh} that this pushforward property coincides with the process of dimensional reduction in topological orders.

\begin{expl}
A well-known example of the application of the pushforward property is to compute the factorization homology of $S^1$ with coefficient in a 1-disk algebra $A$. One can use the standard projection from $S^1$ onto $[-1,1]$. Then the pushforward property and the $\otimes$-excision property imply that 
\be \label{eq:hh-A}
\int_{S^1} A \simeq A\otimes_{A\otimes A^\rev} A.
\ee
We would also like to remark that the $\otimes$-excision property is enough for all computations. 
The pushforward property is used in this work only for the purpose of matching the physical intuition of the dimensional reduction in topological orders. 
\end{expl}


\subsection{Factorization homology for stratified surfaces} \label{sec:fh-2}
In this subsection, we consider a very special case of factorization homology of stratified spaces developed in \cite{aft-1,aft-2}. We will use sets of colors instead of an $\infty$-category of basics for simplicity.

\medskip
Let $L=(L_0,L_1,L_2)$, where $L_0,L_1,L_2$ are three sets whose elements are called {\it colors}. We use $C(X)$ to denote the open cone $X\times [0,1)/{X\times\{0\}}$ of a topological space $X$. Note that $\Rb^2$ can be identified with the open cone $C(S^1)$ of the circle.

\begin{defn}
An {\it unoriented stratified surface} is a pair
$(\Sigma,\,\,\Sigma\xrightarrow{\pi}\{0,1,2\})$
where $\Sigma$ is a surface and $\pi$ is a map. The subspace $\Sigma_i := \pi^{-1}(i)$ is called the {\it $i$-stratum} and its connected components are called {\it $i$-cells}. These data are required to satisfy the following conditions:
\bnu
\item $\Sigma_0$ and $\Sigma_0\cup\Sigma_1$ are closed subspaces of $\Sigma$.
\item For each point $x\in \Sigma_1$, there exists an open neighborhood $U$ of $x$ such that
$(U,U\cap\Sigma_1,U\cap\Sigma_0)\cong (\Rb^2,\Rb^1,\emptyset)$.
 \item For each point $x\in \Sigma_0$, there exists an open neighborhood $V$ of $x$ and a finite subset $I\subset S^1$, such that
 $(V,V\cap\Sigma_1,V\cap\Sigma_0)\cong(\Rb^2,C(I)\backslash\{\text{cone point}\},\{\text{cone point}\})$.
\enu
An {\it (oriented) stratified surface} is a such a pair $(\Sigma,\pi)$ together with an orientation of $\Sigma$ as well as an orientation of each 1-cell.
An {\it $L$-stratified surface} is a stratified surface with each $i$-cell colored by an element of $L_i$, $i=0,1,2$.
\end{defn}

\void{
\begin{rem}
In \cite{af}, the stratified surfaces do not have $L$-labels. We introduce this new notion because it allows us to treat the same stratified 2-disks with different labels as different $L$-stratified 2-disks so that we can assign different $i$-disk algebras in $\EV$ to them (see Def.\,\ref{def:coeff} and Example\,\ref{expl:coeff}). This is very natural and important in the application of factorization homology in topological orders.
\end{rem}
}


\begin{expl}  \label{expl:labeled-disk}
We give three important types of $L$-stratified 2-disks and refer to them as {\it standard $L$-stratified 2-disk}.

\begin{itemize}
\item [$(1)$] $\Sigma=\Rb^2$, $\Sigma_2=\Sigma$ colored by an element of $L_2$. See Fig.\,\ref{fig:labeled-disk} (a). We call such an $L$-stratified 2-disk as a colored $\Rb^2$.
\item [$(2)$] $\Sigma=\Rb^2$, $\Sigma_1=\Rb\times \{0\}$ colored by an element of $L_1$,
$\Sigma_2=\Rb^2\backslash\Rb\times \{0\}$ colored by a pair of elements of $L_2$. See Fig.\,\ref{fig:labeled-disk} (b). We call such an $L$-stratified 2-disk a colored $(\Rb^2;\Rb^1)$.
 \item [$(3)$] $\Sigma=\Rb^2$, $\Sigma_0=\{0\}$ is the origin in $\Rb^2$, $\Sigma_1=C(I)\backslash\{0\}$, $\Sigma_2=\Rb^2\backslash C(I)$, where $I\subset S^1$ is finite subset. See Fig.\,\ref{fig:labeled-disk} (c). 
The $i$-cells are colored by elements of $L_i$ for $i=0,1,2$. We call such an $L$-stratified 2-disk a colored $(\Rb^2;C(I))$. Note that there are many choices of orientations on 1-cells.
\end{itemize}
\end{expl}

\begin{rem} \label{rema:gamma}
The data of $\pi$ in a stratified surface $(\Sigma,\pi)$ is quite implicit. We would like to have a convenient notation that can give us some partial but important information of $\pi$. A stratified surface $\Sigma$ is equivalent to a surface $\Sigma$ together with a finite oriented graph $\Gamma$, in which the edges in $\Gamma$ are 1-cells and vertices in $\Gamma$ are 0-cells, i.e. $\Gamma=(\Sigma_1,\Sigma_0)$. We denote the stratified surface $\Sigma$ also by $(\Sigma;\Gamma)$ or $(\Sigma; \Sigma_1;\Sigma_0)$, i.e.
$$
\Sigma=(\Sigma;\Gamma)=(\Sigma; \Sigma_1;\Sigma_0).
$$
We have already used this notation convention in Example\,\ref{expl:labeled-disk}.
\end{rem}

\begin{defn}
We define $\text{Mfld}^{L-\stra}$ to be the topological category whose objects are $L$-stratified surfaces, and morphism space between two $L$-stratified surfaces $M$ and $N$ are embeddings $e:M\to N$ that preserve the stratifications, the orientations on 1-,2-cells and the colors. We define $\EM\text{fld}^{L-\stra}$ to be the symmetric monoidal $\infty$-category associated to the topological category $\text{Mfld}^{L-\stra}$. The symmetric monoidal structure is given by disjoint union.
\end{defn}

\void{
\begin{defn}
$\text{Disk}^{L-\stra}$ is the topological category for which an object is a disjoint union of standard $L$-stratified 2-disks. For two such $L$-stratified surfaces $M$ and $N$, the morphism space $\text{Hom}_{\text{Disk}^{L-\stra}}(M,N)$ is the space of all embeddings $e:M\to N$ preserving the stratifications, the orientations on 1-,2-cell and the labels. Such an embedding is called an {\it $L$-stratified embedding}. We define $\Disk^{L-\stra}$ to be the symmetric monoidal $\infty$-category associated to the topological category $\text{Disk}^{L-\stra}$ \cite{lurie0}. The symmetric monoidal structure is given by disjoint union.
\end{defn}
}

\begin{defn}
Let $M$ be a stratified surface and let $L_i$ be the set of $i$-cells of $M$ for $i=0,1,2$. We view $M$ as an $L$-stratified surface with each $i$-cell colored by itself. We define $\Disk_M^\stra$ to be the full subcategory of $\EM\text{fld}^{L-\stra}$ consisting of those disjoint unions of standard $L$-stratified 2-disks that admit at least one morphism to $M$.
\end{defn}

\begin{defn}  \label{def:coeff}
Let $\EV$ be a symmetric monoidal $\infty$-category. A {\it coefficient system} on a stratified surface $M$ is a
symmetric monoidal functor $A:\Disk_M^\stra\to \EV$.
\end{defn}

It turns out that a coefficient system $A$ assigns to each $i$-cell of $M$ an $i$-disk algebras in $\EV$. These 0-,1-,2-disk algebras must satisfy additional compatibility conditions in order for $A$ to be a well-defined symmetric monoidal functor. These 0-,1-,2-disk algebras are called the {\it target labels} of $A$. They determine $A$ up to isomorphism. Therefore, we often use target labels to specify a coefficient system in figures.

\begin{figure}
\centerline{
\begin{tabular}{@{}c@{\quad\quad}c@{\quad\quad}c@{}}
\begin{picture}(80, 95)
   \put(10,10){\scalebox{2}{\includegraphics{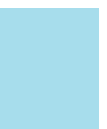}}}
   \put(-30,-54){
     \setlength{\unitlength}{.75pt}\put(-18,-19){
     \put(100, 150)    {$\EA$}
   }\setlength{\unitlength}{1pt}}
  \end{picture}
  &
\begin{picture}(90, 95)
   \put(10,10){\scalebox{2}{\includegraphics{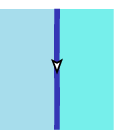}}}
   \put(-30,-54){
     \setlength{\unitlength}{.75pt}\put(-18,-19){
     \put(85, 150)    {$\EA$}
     \put(137, 150)    {$\EB$}
     \put(108,206)     { $\EM$}
   }\setlength{\unitlength}{1pt}}
  \end{picture}
&
  \begin{picture}(150,100)
   \put(10,0){\scalebox{2.2}{\includegraphics{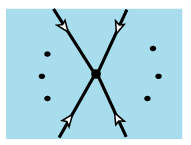}}}
   \put(10,0){
     \setlength{\unitlength}{.75pt}\put(-8,-9){
     \put(93,68)   { $\EP$}
     \put(118, 108)     {$ \EM_1 $}
     \put(37, 108)     {$ \EM_n $}
     \put(36,30)      {$ \EM_i$}
     \put(118,30)      {$\EM_{i-1}$}
     \put(81,115)      {$\EA_0$}
     \put(78,25)      {$\EA_i$}
          }\setlength{\unitlength}{1pt}}
  \end{picture}  \\
  (a) & (b) & (c)
\end{tabular}}
\caption{For a given coefficient system $A$, the values of $A$ on three types of $L$-stratified 2-disks are shown in figures (a), (b) and (c) and discussed in details in Example\,\ref{expl:coeff}. In particular, in figure (c), 2-disk algebras $\EA_0,...,\EA_n=\EA_0$ are assigned to 2-cells, 1-disk algebras $\EM_1, \cdots, \EM_n$ are assigned to 1-cells, and a 0-disk algebra $\EP$ is assigned to the unique 0-cell.}
\label{fig:labeled-disk}
\end{figure}

\begin{expl}  \label{expl:coeff}
We give some examples of the values of the coefficient system $A$ on a few typical objects in $\Disk_M^\stra$.
\bnu
\item The functor $A$ assigns to each colored $\Rb^2$ (see Fig.\,\ref{fig:labeled-disk} (a)) a 2-disk algebra $\EA$ in $\EV$. Moreover, differently colored $\Rb^2$ can be assigned to (not necessarily) different 2-disk algebras in $\EV$.

\item The functor $A$ assigns to each colored $(\Rb^2;\Rb^1)$ (see Fig.\,\ref{fig:labeled-disk} (b)) a 1-disk algebra $\EM$, which is required to have compatible two-side actions of $\EA$ and $\EB$, where $\EA$ and $\EB$ are the 2-disk algebras assigned to the adjacent 2-cells (see Fig.\,\ref{fig:labeled-disk} (b)). Note that our convention of left and right is that if one stands on the 1-cell seeing the arrow pointing towards you, then the left hand side of you is treated as the left. For example, $\EA$ in Fig.\,\ref{fig:labeled-disk} (b) acts on $\EM$ from left. 
In this case, the relation between the 1-disk algebra $\EM$ and the 2-disk algebras $\EA$ and $\EB$ was explained in \cite[Sec.\,2.3]{bbj2} (see also \cite[Prop.\,4.8]{aft-2}). More precisely, let $\EV=\EV_{\mathrm{uty}}$, let $\EA$ and $\EB$ be UMTC's and let $\EM$ be a UMFC. In this case, the compatible two-side actions on $\EM$ of $\EA$ and $\EB$ means that $\EM$ is a multi-fusion $\EA$-$\EB$-bimodule.  

\item The functor $A$ assigns to each colored $(\Rb^2;C(I))$ (see Fig.\,\ref{fig:labeled-disk} (c)) a 0-disk algebra $\EP$ equipped with actions from 1-disk algebras $\EM_1, \cdots, \EM_n$ (assigned to the adjacent 1-cells) and actions from 2-disk algebras $\EA_0, \cdots, \EA_{n-1}$ (assigned to the adjacent 2-cells). A precise statement of these actions is given in Example\,\ref{expl:fh-labeled-disk} (3). 
\enu
\end{expl}


\begin{defn}
Let $\EV$ be a symmetric monoidal $\infty$-category, 
$M$ be a stratified surface, 
and $A:\Disk_M^\stra\to \EV$ be a coefficient system.
The {\it factorization homology} of $M$ with coefficient in $A$ is an object of $\EV$ defined as follows:
$$\int_MA:=\text{Colim}\left((\Disk_M^\stra)_{/M}\to\Disk_{M}^\stra\xrightarrow{A}\EV\right).$$
\end{defn}

\begin{expl} \label{expl:fh-labeled-disk}
We give a few examples of factorization homology.
\begin{itemize}
\item [$(1)$] Let $M_1$ be a stratified 2-disk as depicted in Fig.\,\ref{fig:labeled-disk} (a). If a coefficient system $A_1$ assigns a 2-disk algebra $\EA$ to $M_1$, we have $\int_{M_1}A_1\simeq \EA$ and $\int_{M_1\backslash\{0\}} A_1\simeq HH_{\ast}(\EA)$, where $HH_\ast(\EA)$ is the usual Hochschild homology, i.e. $HH_{\ast}(\EA)=\EA\otimes_{\EA\otimes \EA^\op} \EA$.

\item [$(2)$] Let $M_2$ be a stratified 2-disk as depicted in Fig.\,\ref{fig:labeled-disk} (b). If a coefficient system $A_2$ assigns a 1-disk algebra $\EM$ to the 1-cell. We have $\int_{M_2}A_2\simeq \EM$.

\item [$(3)$] Let $M_3$ be a stratified 2-disk as depicted in Fig.\,\ref{fig:labeled-disk} (c). Suppose a coefficient system $A_3$ assigns a 0-disk algebra $\EP$ to the 0-cell, and assigns $i$-disk algebras to the adjacent $i$-cells for $i=1,2$ as shown in Fig.\,\ref{fig:labeled-disk} (c). Then we have $\int_{M_3}A_3\simeq \EP$ and
\be \label{eq:hh-act-on-P}
\int_{M_3\backslash\{0\}} A_3 \simeq HH_{\ast}(\EA_0,\EM_1\otimes_{\EA_1}...\otimes_{\EA_{n-1}}\EM_n)\simeq \otimes_{\EA_0, \cdots, \EA_{n-1}}^\circlearrowright (\EM_1, \cdots, \EM_n),
\ee
where $\EM_1\otimes_{\EA_1}...\otimes_{\EA_{n-1}}\EM_n$ is an $\EA_0$-$\EA_0$-bimodule, and $HH_{\ast}(\EA_0,\EM_1\otimes_{\EA_1}...\otimes_{\EA_{n-1}}\EM_n)$ is the Hochschild homology of this bimodule. Since $M_3\backslash \{0\} \simeq S^1\times \Rb$, $\int_{M_3\backslash\{0\}} A_3$ has a structure of 1-disk algebra and acts on $\EP$. Our convention is that if the arrows point towards the 0-cell, then we set the action to be the left action. If the arrow on the 1-cell with target label $\EM_i$ in Fig.\,\ref{fig:labeled-disk} (b) is flipped, we need replace $\EM_i$ in Eq.\,(\ref{eq:hh-act-on-P}) by $\EM_i^\rev$, where $\EM_i^\rev$ denotes the opposite 1-disk algebra of $\EM_i$. Therefore, the 0-disk algebra $\EP$ receives a left $\int_{M_3\backslash\{0\}} A_3$-module structure. This module structure condition on $\EP$ is sufficient for the coefficient system $A_3$ to be well-defined, due to the fact that $M_3\backslash\{0\}$ is universal among all embeddings from standard stratified 2-disks without 0-cell into $M_3$. A nice explanation of this fact in a special case ($\EM_i=\EA_j$ for all $i,j$) was given in \cite[Sec.\,3]{bbj2} in terms of the so-called braided module category (see also \cite{ginot}). 

\end{itemize}
\end{expl}

For the factorization homology of stratified manifolds, there is also a version of $\otimes$-excision property and pushforward property \cite[Thm.\,2.25, Cor.\,2.40]{aft-2}. More precisely, Thm.\,\ref{thm:tensor-excision} remains true if the map $M\to[-1,1]$ restricts to a bundle of stratified manifolds over $(-1,1)$. Eq.\,(\ref{eq:pushforward}) is also true for stratified manifolds $M$ and $N$ if the map $f:M\to N$ restricts to a bundle of stratified manifolds over each cell of $N$.

\begin{figure}
\centerline{
\begin{picture}(140, 95)
   \put(10,10){\scalebox{2}{\includegraphics{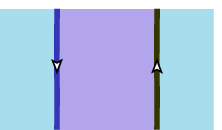}}}
   \put(-30,-54){
     \setlength{\unitlength}{.75pt}\put(-18,-19){
     \put(85, 150)    {$\EC$}
     \put(145, 150)    {$\ED$}
     \put(210, 150)    {$\EE$}
     \put(108,206)     { $\EM$}
     \put(185,206)     { $\EN$}
   }\setlength{\unitlength}{1pt}}
  \end{picture}
}
\caption{A stratified 2-disk $K=(\Rb^2; \Rb \cup \Rb)$ with a coefficient system $A_K$ determined by 2-disk algebras $\EC,\ED,\EE$ for 2-cells and 1-disk algebras $\EM,\EN$ for 1-cells.}
\label{fig:af-1}
\end{figure}

\begin{expl} \label{expl:otimes}
We illustrate the $\otimes$-excision property by a few examples that are sufficient for the purpose of computing factorization homology.
\bnu

\item A stratified 2-disk $K$ with a coefficient system $A_K$ is shown in Fig.\,\ref{fig:af-1}. The 1-disk algebra $\EM$ is a $\EC$-$\ED$-bimodule, and $\EN$ is an $\EE$-$\ED$-bimodule according to our left-right convention.
By the $\otimes$-excision property, we have
\be \label{eq:fh-otimes}
\int_K A_K \simeq \EM \otimes_\ED \EN^\rev.
\ee
Consider a process of fusing these two 1-cells labeled by $\EM$ and $\EN$ into one with a downward arrow and labeled by $\EM\otimes_\ED \EN^\rev$. This process produces a new stratified 2-disk $K'$ and a new coefficient system $A_{K'}$ determined by $\EC,\EM\otimes_\ED \EN^\rev, \EE$. We have $\int_K A_K \simeq \int_{K'} A_{K'}$.

\begin{figure}[bt]
\centerline{
\begin{tabular}{@{}c@{\quad\quad}c}
\raisebox{-0pt}{
  \begin{picture}(200,100)
   \put(20,0){\scalebox{2.2}{\includegraphics{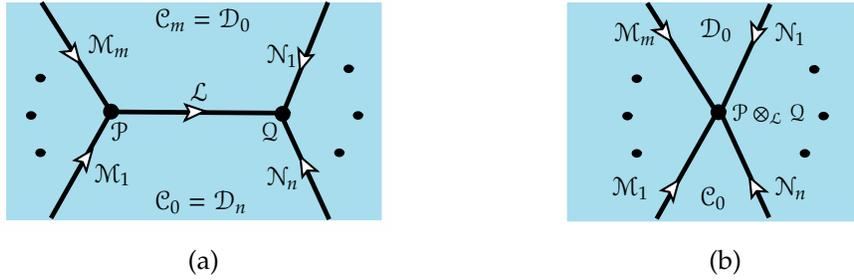}}}
   \put(20,0){
     \setlength{\unitlength}{.75pt}\put(-8,-9){
     \put(145,97)  {$\EN_1$}
     \put(103,78)   { $\EL$}
     \put(85,115) { $\EC_m = \ED_0$}
     \put(85,23)   { $\EC_0 = \ED_{n}$}
     \put(55, 102)     {$ \EM_m $}
     \put(58, 37)     {$ \EM_1 $}
     \put(66,58)      {$\EP$}
     \put(142,58)    {$\EQ$}
     \put(145,35)      {$\EN_n$}
          }\setlength{\unitlength}{1pt}}
  \end{picture}}
&  
\raisebox{-0pt}{
  \begin{picture}(150,100)
   \put(10,0){\scalebox{2.2}{\includegraphics{pic-trees-2.eps}}}
   \put(10,0){
     \setlength{\unitlength}{.75pt}\put(-8,-9){
     \put(80,110)  {$\ED_0$}
     \put(80,25)   {$\EC_0$}
     \put(93,68)   { \small{$\EP\otimes_\EL\EQ$}}
     \put(118, 108)     {$ \EN_1 $}
     \put(37, 108)     {$ \EM_m $}
     \put(36,30)      {$\EM_1$}
     \put(118,30)      {$\EN_n$}
          }\setlength{\unitlength}{1pt}}
  \end{picture}}  \\
  (a) & (b)
\end{tabular}}
\caption{These two figures depict a process of contracting a 1-cell to a 0-cell with proper new target labels such that the value of factorization homology is not changed. Figure (a) depicts a stratified 2-disk $M$ with a coefficient system $A_M$ determined by its target labels; figure (b) depicts another stratified 2-disk $M'$ with a coefficient system $A_{M'}$ determined by its target labels. $M'$ is obtained by contracting the internal edge labeled by $\EL$ in Figure (a) to a point.
}
\label{fig:trees}
\end{figure}

\item Let $M$ be a stratified 2-disk with a coefficient system $A_M$ as shown Fig.\,\ref{fig:trees} (a). In particular, $\EP$ and $\EQ$ are 0-disk algebras in $\EV$; $\EM_i$ and $\EN_j$ are 1-disk algebras; and $\EC_k$ are 2-disk algebras. It is easy to realize $M$ as a collar-gluing of three pieces $M_{[-1,1)}$, $M_{\{0\}}\times \Rb$ and $M_{(-1,1]}$ (recall Def.\,\ref{def:collar}), where $M_{\{0\}}$ is given by $(\Rb,\{0\})$ (a vertical line intersecting the internal edge labeled by $\EL$ in Fig.\,\ref{fig:trees} (a)). According to the $\otimes$-excision property, we have
$$
\int_MA_M\simeq \EP\otimes_{\EL}\EQ.
$$
Consider a process of contracting the internal edge labeled by $\EL$ to a point. It produces a new stratified 2-disk $M'$ with a new coefficient system $A_{M'}$ depicted in Fig.\,\ref{fig:trees} (b).
We have
$$
\int_MA_M \simeq \EP\otimes_{\EL}\EQ \simeq \int_{M'} A_{M'}.
$$

\item 
Let $N$ be a stratified 2-disk with a coefficient system $A_N$ as depicted in Fig.\,\ref{fig:trees} (a). By the same argument as the previous case, we have $$
\int_N A_N \simeq \EP \otimes_{\EK\otimes_\EE \EL} \EQ.
$$
Therefore, $\int_N A_N \simeq \int_{N'} A_{N'}$ where $N'$ is a stratified 2-disk with a coefficient system $A_{N'}$ as depicted in Fig.\,\ref{fig:trees} (b).

\void{
The coefficient system $A_N$ provides a stratified 2-disk $N$ depicted in Fig.\,\ref{fig:loop} (a) with the target labels. To compute the factorization homology $\int_NA$, the $\otimes$-excision property allows us to reduce (a) to (b) with a proper target label $(\EP,p)$ on the unique 0-cell as explained by above example. Moreover, we can squeeze the 2-disk encircled by the loop labeled by $\EM$ in (b) vertically to obtain a 1-cell in a new stratified 2-disk $N'$ depicted in (c). A coefficient system $A_{N'}'$ on $N'$ is determined by the new target labels $\EE=\EM\otimes_\ED \EM^\rev$ and $\EM$, which is the 0-disk algebra obtained by forgetting its 1-disk algebra structure on $\EM$. 
According to the $\otimes$-excision property again, we have
$$
\int_N A_N \simeq \int_{N'} A_{N'}' \simeq \EP \otimes_\EE \EM.
$$

\begin{figure}[bt]
\centerline{
\begin{tabular}{@{}c@{\quad\quad}c@{\quad\quad}c}
\raisebox{-0pt}{
  \begin{picture}(95,100)
   \put(0,0){\scalebox{2}{\includegraphics{pic-loop.eps}}}
   \put(0,0){
     \setlength{\unitlength}{.75pt}\put(-8,-9){
     \put(120,30)  {$\EC$}
     \put(70, 63)     {$ \ED $}
          }\setlength{\unitlength}{1pt}}
  \end{picture}}
  &
  \raisebox{-0pt}{
  \begin{picture}(95,100)
   \put(0,0){\scalebox{2}{\includegraphics{pic-loop-1.eps}}}
   \put(0,0){
     \setlength{\unitlength}{.75pt}\put(-8,-9){
     \put(75,75)   { $\ED$}
     \put(115, 100)     {$ \EC $}
     \put(55,57)      {$\EP$}
     \put(110,39)      {$\EM$}
          }\setlength{\unitlength}{1pt}}
  \end{picture}}
  &
  \raisebox{-0pt}{
  \begin{picture}(95,100)
   \put(0,0){\scalebox{2}{\includegraphics{pic-loop-2.eps}}}
   \put(0,0){
     \setlength{\unitlength}{.75pt}\put(-8,-9){
     \put(100,51)   { $\EM$}
     \put(115, 100)     {$ \EC $}
     \put(47,72)      {$\EP$}
     \put(73,51)      {$\EE$}
          }\setlength{\unitlength}{1pt}}
  \end{picture}}  \\
  (a) & (b) & (c)
\end{tabular}}
\caption{These figures show the processes of reducing a graph with a contractible loop to a graph without loop and providing proper new target labels such that the value of the factorization homology is not changed. In (a), we have a contractible loop with finitely many vertices and external legs. By retracting internal edges on the circle (as depicted in Fig.\,\ref{fig:trees}), we reduce (a) to (b), in which we make some target labels explicit. We obtained (c) by squeezing the 2-disk encircled by the loop in (b) to a line with a new target label given by $\EE=\EM\otimes_\ED \EM^\rev$ and an end point labeled by $\EM$ viewed as a 0-disk algebra.}
\label{fig:loop}
\end{figure}
}

\begin{figure}[bt]
\centerline{
\begin{tabular}{@{}c@{\quad\quad}c}
\raisebox{-0pt}{
  \begin{picture}(200,100)
   \put(20,0){\scalebox{2.2}{\includegraphics{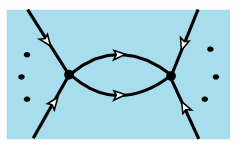}}}
   \put(20,0){
     \setlength{\unitlength}{.75pt}\put(-8,-9){
     \put(145,97)  {$\EN_1$}
     \put(105,95)   { $\EL$}
     \put(103,68)   { $\EE$}
     \put(105,40)   { $\EK$}
     \put(55, 102)     {$ \EM_m $}
     \put(58, 37)     {$ \EM_1 $}
     \put(64,55)      {$\EP$}
     \put(144,55)    {$\EQ$}
     \put(145,35)      {$\EN_n$}
          }\setlength{\unitlength}{1pt}}
  \end{picture}}
&  
\raisebox{-0pt}{
  \begin{picture}(200,100)
   \put(20,0){\scalebox{2.2}{\includegraphics{pic-trees.eps}}}
   \put(20,0){
     \setlength{\unitlength}{.75pt}\put(-8,-9){
     \put(145,97)  {$\EN_1$}
     \put(83,80)   { $\EK\otimes_\EE \EL$}
     \put(55, 102)     {$ \EM_m $}
     \put(58, 37)     {$ \EM_1 $}
     \put(66,58)      {$\EP$}
     \put(142,58)    {$\EQ$}
     \put(145,35)      {$\EN_n$}
          }\setlength{\unitlength}{1pt}}
  \end{picture}} \\
  (a) & (b)
\end{tabular}}
\caption{These two figures depict a process of merging two 1-cells and one 2-cell to a single 1-cell with proper new target labels such that the value of factorization homology is not changed.
}
\label{fig:loop}
\end{figure}

\item The $\otimes$-excision property also allows us to add 0-cells and 1-cells with proper target labels without changing the value of factorization homology.
\bnu

\item On any 1-cell $e$ with the target label $\EM$, we can add a new 0-cell labeled by the 0-disk algebra $\EM$, which is obtained by forgetting its 1-disk algebra structure. Now $e$ breaks into two 1-cells both labeled by $\EM$ and connected by a 0-cell labeled by the 0-disk algebra $\EM$.

\item Between any two 0-cells $p$ and $q$ (not necessarily distinct) on the boundary of a given 2-cell labeled by $\EC$, we can add an oriented 1-cell within this 2-cell from $p$ to $q$ labeled by the 1-disk algebra $\EC$ obtained by forgetting its 2-disk algebra structure.

\enu
These two ways of adding 0-,1-cells with proper target labels allows us to break two adjacent non-contractible loops on a surface into two loops, each of which can be covered by an open 2-disk. We will use this fact in the proof of Thm.\,\ref{thm:higher-g-defects}.

\enu
We have described the application of $\otimes$-excision property to three types of stratified 2-disks. More generally, let $X$ be one of the three stratified 2-disks with the coefficient system $A_X$ described above (i.e. $X=K,M,N$). Let $Y$ be any stratified surface equipped with a coefficient system $A_Y$ such that there is a stratified embedding $f: X \hookrightarrow Y$ that is compatible with the two coefficient systems $A_X$ and $A_Y$. We can replace the image of $f$ in $Y$ by the stratified 2-disk $X'$ (i.e. $X'=K',M',N'$) to obtain a new stratified surface $Y'$ which is identical to $Y$ outside $X'$, i.e. $Y'\backslash X' \cong Y\backslash X$. $Y'$ is equipped with a canonical coefficient system $A_{Y'}$ which is identical to $A_Y$ on $Y'\backslash X'$, and is identical to the coefficient system $A_{X'}$ on $X'$. Then the $\otimes$-excision property implies that $\int_Y A_Y \simeq \int_{Y'} A_{Y'}$.
\end{expl}

\section{Computation of factorization homology}  \label{sec:fh-comp}

In this section, we introduce the notion of an anomaly-free coefficient system and prove our main result.

\subsection{Anomaly-free coefficient systems}

Let $M$ be a stratified surface.
Let $\EV_{\mathrm{uty}}$ be the symmetric monoidal $(2,1)$-category of unitary categories. The following definition is motivated by the notion of an anomaly-free topological order.
\begin{defn} \label{def:af-cond}
A coefficient system $A:\Disk_M^\stra\to \EV_{\mathrm{uty}}$ on $M$ is called {\it anomaly-free} if the following conditions are satisfied:
\bnu
\item The target label for a 2-cell is given by a UMTC;
\item The target label for a 1-cell between two adjacent 2-cells labeled by $\EA$ (left) and $\EB$ (right) is given by a closed multi-fusion $\EA$-$\EB$-bimodule (thus $\EA$ and $\EB$ are Witt equivalent);
\item The target label for a 0-cell as the one depicted in Fig.\,\ref{fig:labeled-disk} (c) is given by a 0-disk algebra $(\EP,p)$, where the unitary category $\EP$ is equipped with the structure of a closed left $\int_{M\backslash\{0\}} A$-module (recall Def.\,\ref{def:closed-module}), i.e.
\be \label{eq:af-cond-3}
\int_{M\backslash\{0\}} A \simeq \cboxtimes_{\EA_0, \cdots,\EA_{n-1}}(\EM_1, \cdots, \EM_n) \simeq \fun_\bk(\EP,\EP).
\ee
According to Cor.\,\ref{cor:0-cell-condition}, $\EP$ is uniquely determined (up to equivalence) by other data.
\enu
\end{defn}

\begin{figure}[bt]
$$
\raisebox{-60pt}{ \begin{picture}(120, 130)
   \put(-40,0){\scalebox{2}{\includegraphics{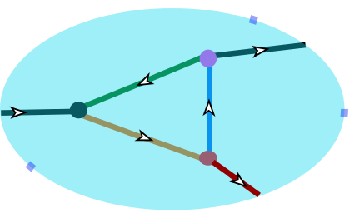}}}
   \put(1,26){
     \setlength{\unitlength}{.75pt}\put(-18,-19){
     \put(-5, 75)        { $(\EP,p) $}
     \put(-28,70)       { $\EI$ }
     \put(93, 10)       { $(\ER,r)$}
     \put(102, 115)      { $(\EQ,q)$}
     \put(154, 118)     {  $\EJ$ }
     \put(64, 28)     { $ \EL $}
     \put(78, 59)     { $ \EC $}
     \put(56, 92)     { $ \EM $}
     \put(128, 62)     { $\EN$ }
     \put(148, 13)     { $\EK$ }
     \put(20,5)      { $\EA$ }
     \put(209,65)   { $\EB$ }
     \put(20,115)   { $\ED$ }
     }\setlength{\unitlength}{1pt}}
  \end{picture}}
$$
\caption{This figure depicts a stratified 2-disk with an anomaly-free coefficient system $A$ determined by its target labels.}
\label{fig:defects}
\end{figure}

\begin{expl}
A stratified 2-disk $M$ is depicted in Fig.\,\ref{fig:defects}. An anomaly-free  coefficient system $A$ on $M$ is determined by its target labels as shown in the figure.
\bnu
\item The target labels for 2-cells: $\EA, \EB,\EC,\ED$ are Witt equivalent UMTC's.

\item The target labels for 1-cells: $\EL$ is a closed multi-fusion $\EA$-$\EC$-bimodule, $\EM$ is a closed multi-fusion $\ED$-$\EC$-bimodule and $\EN$ is a closed multi-fusion $\EB$-$\EC$-bimodule.

\item The target labels for 0-cells:
\bnu
\item $(\EP,p)$ is a closed left module over $\cboxtimes_{\EA,\ED,\EC}(\EI,\EM,\EL^\rev)$;
\item $(\EQ,q)$ is a closed left module over
$\cboxtimes_{\EC,\ED,\EB}(\EM^\rev,\EJ^\rev,\EN)$;
\item $(\ER,r)$ a closed left module over
$\cboxtimes_{\EA,\EC,\EB}(\EL,\EN^\rev,\EK^\rev)$.
\enu
\enu
\end{expl}

We would like to express the data of a coefficient system $A:\Disk_M^\stra\to \EV_{\mathrm{uty}}$ more explicitly. For example, in the case of a stratified 2-disk depicted in Fig.\,\ref{fig:defects}, we denote
$$
A=(\EA,\EB,\EC,\ED; \EI,\EJ,\EK,\EL,\EM,\EN; (\EP,p),(\EQ,q),(\ER,r)).
$$
If there is no $i$-cell, we simply denote the target label for $i$-cell by $\emptyset$. For example, if $M$ is a stratified surface without 1-,0-strata and the target label for the unique 2-cell is $\EC$, we denote $A$ by $(\EC;\emptyset;\emptyset)$.


\begin{rem} \label{rem:0-cell-label}
We would like to point out that the target labels for 0-cells can be more general than the ones appeared in Fig.\,\ref{fig:defects}. We have seen that two 0-cells in a stratified surface $\Sigma$ can be labeled by $(\EP,p)$ and $(\EQ,q)$, respectively. Let $p_1,\cdots, p_k\in \EP$ and $q_1,\cdots, q_k\in \EQ$. Suppose $\int_{\Sigma} (\cdots;\cdots;\cdots, (\EP,p_i), (\EQ,q_i)) = (\EX,x_i)$. We allow the two 0-cells to be simultaneously labeled by $(\EP\boxtimes \EQ, \oplus_{i=1}^k p_i\boxtimes q_i)$. The factorization homology with coefficients including two 0-cells simultaneously labeled by $(\EP\boxtimes \EQ, \oplus_{i=1}^k p_i\boxtimes q_i)$ is defined as follows:
$$
\int_{\Sigma} (\cdots;\cdots;\cdots,(\EP\boxtimes \EQ, \oplus_i p_i\boxtimes q_i)) := (\EX, \oplus_{i=1}^k x_i).
$$
\end{rem}

\void{

\begin{thm}[Existence of the factorization homology] \label{thm:existence}
The {\it factorization homology} of a stratified surface $M$ with an anomaly-free coefficient system $A$ is well-defined as an object of $\EV_{\mathrm{uty}}$. Namely, the following Colim 
$$\int_MA:=\text{Colim}\left((\Disk_M^\stra)_{/M}\to\Disk_{M}^{L-\stra}\xrightarrow{A}\EV_{\mathrm{uty}}\right)$$
exists in $\EV_{\mathrm{uty}}$. 
\end{thm}
\pf

\epf

}

\void{
The following lemma is very useful when we compute factorization homology in Sec.\,\ref{sec:fh-comp}.
\begin{lem}
If we flip the orientations of some 1-cell $e_1, \cdots, e_k$, we obtain a new oriented graph $\Gamma'$. If we also replace their label $\EM_{e_1}, \cdots, \EM_{e_k}$ by $\EM_{e_1}^\rev, \cdots, \EM_{e_k}^\rev$, respectively, and modify 0-cell labels accordingly (see Rem.\,\ref{rem:flip-arrow}), we obtain a new coefficient system $\mathrm{Coeff}_{\Sigma'}'$. We must have
$$
\int_{\Sigma} \mathrm{Coeff}_{\Sigma}
=\int_{\Sigma'} \mathrm{Coeff}_{\Sigma'}'.
$$
\end{lem}
\pf
This following from the $\otimes$-excision property immediately.
\epf
}

\subsection{A few useful mathematical results}

Let $\EC,\ED$ be UMFC's throughout this subsection.

\begin{defn}
Let $\EM$ be a $\EC$-$\ED$-bimodule. We say that $\EM$ is {\em right dualizable}, if there exists a $\ED$-$\EC$-bimodule $\EN$ equipped with bimodule functors $u:\ED\to\EN\boxtimes_\EC\EM$ and $v:\EM\boxtimes_\ED\EN\to\EC$ such that the composed bimodule functors
$$\EM \simeq \EM\boxtimes_\ED\ED \xrightarrow{\id_\EM\boxtimes u} \EM\boxtimes_\ED\EN\boxtimes_\EC\EM \xrightarrow{v\boxtimes\id_\EM} \EC\boxtimes_\EC\EM \simeq \EM,$$
$$\EN \simeq \ED\boxtimes_\ED\EN \xrightarrow{u\boxtimes\id_\EN} \EN\boxtimes_\EC\EM\boxtimes_\ED\EN \xrightarrow{\id_\EN\boxtimes v} \EN\boxtimes_\EC\EC \simeq \EN$$
are isomorphic to the identity functors. In this case, the $\ED$-$\EC$-bimodule $\EN$ is called {\it left dualizable}. The category $\EN$ is called the {\it right dual} of $\EM$, and $\EM$ is called the {\it left dual} of $\EN$.
\end{defn}

\begin{rem}
The right/left dual of a $\EC$-$\ED$-bimodule $\EM$, if exists, is unique up to equivalence.
\void{
in the following sense. Let $(\EN',u':\ED\to\EN'\boxtimes_\EC\EM,v':\EM\boxtimes_\ED\EN'\to\EC)$ be another triple as in the definition. Then there is an equivalence $F:\EN\to\EN'$ of $\ED$-$\EC$-bimodules such that the induced diagrams
$$
\xymatrix{
  & \ED \ar[ld]_{u} \ar[rd]^{u'} \\
  \EN\boxtimes_\EC\EM \ar[rr]^{F\boxtimes\id_\EM} && \EN'\boxtimes_\EC\EM \\
}
\quad\quad
\xymatrix{
  & \EC \\
  \EM\boxtimes_\ED\EN \ar[ru]^{v} \ar[rr]^{\id_\EM\boxtimes F} && \EM\boxtimes_\ED\EN' \ar[lu]_{v'} \\
}
$$
are commutative up to isomorphism. Moreover, the equivalence $F$ is unique up to isomorphism.
}
\end{rem}

The existence of the right/left dual of a $\EC$-$\ED$-bimodule $\EM$ is known \cite{dss}. We give a short proof of this result. In particular, we make the duality functors $u$ and $v$ explicit for the later uses. 
\begin{thm}  \label{thm:dual-mod-cat}
A $\EC$-$\ED$-bimodule $\EM$ has a right dual given by $\EM^\op$ with two duality maps $u$ and $v$ defined as follows:
\begin{align}
u:\,\, &\ED \to \fun_\EC(\EM,\EM)\simeq \EM^\op\boxtimes_\EC\EM, \quad\quad  d\mapsto  -\odot d,    \nn
v:\,\, &\EM\boxtimes_\ED\EM^\op\to \EC,    \quad\quad\quad\quad         x\boxtimes_\ED y \mapsto [x,y]_\EC^\ast.
\label{eq:v}
\end{align}
Since $(\EM^\op)^\op \simeq \EM$ as $\EC$-$\ED$-bimodules, $\EM^\op$ is also the left dual of $\EM$.
\end{thm}

\begin{proof}
It is clear that $u$ is a $\ED$-$\ED$-bimodule functor. That $v$ is a $\EC$-$\EC$-bimodule functor follows from the identity Eq.\,(\ref{eq:internal-hom}).
To show that $\EM^\op$ is the right dual of $\EM$, first, we consider the following diagram:
\be  \label{eq:duality-1}
\xymatrix{
\EM \simeq \EM\boxtimes_\ED\ED \ar[rr]^-{\id_\EM\boxtimes_\ED u} \ar[drr]_{\id_\EM\boxtimes_\ED u} & & \EM\boxtimes_\ED\EM^\op\boxtimes_\EC\EM \ar[rr]^-{v\boxtimes_\EC\id_\EM} \ar[d]^\simeq & & \EC\boxtimes_\EC\EM \simeq \EM\, . \\
& & \EM\boxtimes_\ED \Fun_\EC(\EM, \EM) \ar[urr]_{x\boxtimes_\ED f \mapsto f(x)}  \ar[u] & &
}
\ee
It is commutative due to Eq.\,(\ref{eq:xRy}). It follows immediately that the composed functor on the first row is isomorphic to $\id_\EM$. Secondly, we consider the following composed functor
\begin{align}
\EM^\op \simeq \ED\boxtimes_\ED\EM^\op \xrightarrow{u\boxtimes_\ED\id_{\EM^\op}} \EM^\op\boxtimes_\EC\EM\boxtimes_\ED\EM^\op  \xrightarrow{\id_{\EM^\op}\boxtimes_\EC v} \EM^\op\boxtimes_\EC\EC \simeq \EM^\op.  \label{eq:duality-2}
\end{align}
Notice that the functor $\id_{\EM^\op}\boxtimes_\EC v$ maps $x\boxtimes_\EC y \boxtimes_\ED z$ to $x\odot^\ast [y,z]^\ast=[y,z]\odot x$. 
According to \cite[Eq.\,(2.3)]{kong-zheng}, the functor $[y,-]\odot x \in \fun_\EC(\EM, \EM)$ is right adjoint to $[-,x]^\ast\odot y$ for $x,y\in\EM$.
Therefore, the composed functor (\ref{eq:duality-2}) is adjoint, hence isomorphic, to the identity functor. 
This shows $\EM^\op$ is the right dual of $\EM$.
\end{proof}

\begin{rem}
With a proper adaption of the left/right duals, the same proof works for $\EC,\ED$ being finite multi-tensor categories and $\EM$ being finite $\EC$-$\ED$-bimodules. 
\end{rem}

\begin{rem} \label{rem:inv-bimod}
A $\EC$-$\ED$-bimodule $\EM$ is called {\it invertible} if both of $u$ and $v$ are bimodule equivalences. In the special case where $\EC,\ED$ are indecomposable, if one of $u$ and $v$ is a bimodule equivalence, so is the other \cite[Prop.\,4.2]{eno2009}. 
\end{rem}

\void{
\begin{lem}
Let $\EM$ be a closed multi-fusion $\EC$-$\ED$-bimodule. We have a canonical monoidal equivalence
\be
\EM \boxtimes_\ED \EM^\rev \xrightarrow{\simeq} \fun_\EC(\EM, \EM) \quad\quad
\mbox{defined by} \quad\quad x\boxtimes_\ED y \mapsto x\otimes - \otimes y.
\ee
\end{lem}
\pf
We have the following natural equivalences
$$
\EC\boxtimes_{Z(\EC)} \left( \EM\boxtimes_\ED \EM^\rev \right) \simeq \EC\boxtimes_{\EC\boxtimes \overline{\EC}} \left( \EM\boxtimes_\ED \EM^\rev \right) \simeq \EM \boxtimes_{\overline{\EC} \boxtimes \ED} \EM^\rev \simeq \EM\boxtimes_{Z(\EM)} \EM^\rev \simeq \fun_\bk(\EM, \EM),
$$
where we have used Eq.\,(\ref{eq:CC=FCC}). Moreover, it maps
$$
c\boxtimes_{Z(\EC)} (x\boxtimes_\ED y) \mapsto (c\odot x) \boxtimes_{Z(\EM)} y \mapsto (c\odot x)\otimes - \otimes y.
$$
Therefore, we obtain the following canonical equivalences 
\begin{align} \label{eq:xxx}
\EM\boxtimes_\ED \EM^\rev & \xrightarrow{\simeq} Z(\EC)\boxtimes_{Z(\EC)} \left( \EM\boxtimes_\ED \EM^\rev \right) \simeq \fun_{\EC|\EC}(\EC,\EC)\boxtimes_{Z(\EC)} \left( \EM\boxtimes_\ED \EM^\rev \right) \nn
&\xrightarrow{\simeq} \fun_{\EC|\EC} \left( \EC, \EC\boxtimes_{Z(\EC)}\left( \EM\boxtimes_\ED \EM^\rev\right) \right) \nn
&\xrightarrow{\simeq} \fun_{\EC|\EC} \left( \EC, \fun_\bk(\EM,\EM)\right) \nn
&\xrightarrow{\simeq} \fun_\EC(\EM,\EM)
\end{align}
defined by $x\boxtimes_\ED y \mapsto \id_\EC \boxtimes_\EC (x\boxtimes_\ED y) \mapsto \id_\EC(-) \boxtimes_{Z(\EC)} x\boxtimes_\ED y \mapsto \id_\EC(-) \odot x \otimes - \otimes y \mapsto x\otimes - \otimes y$. The third $\simeq$ in Eq.\,(\ref{eq:xxx}) was proved in \cite[Lem.\, 5.4]{kong-zheng}.  
\epf
}

\subsection{Closed stratified surfaces without 1-stratum} \label{sec:fh-1}
In this subsection, we compute the factorization homology of any closed stratified surfaces without 1-stratum.

\void{
\medskip
When $\EC$ is a UMTC, we have $Z(\EC) \simeq \EC\boxtimes\overline{\EC}$ and $\EC \simeq \EC^\rev$ (via the braidings). Therefore, as a special case of Eq.\,(\ref{eq:CC=FCC}), we obtain a canonical monoidal equivalence:
$$
\EC\boxtimes_{\EC\boxtimes \overline{\EC}}\EC \simeq \fun_\bk(\EC,\EC). 
$$
}

\begin{figure}[bt]
\centerline{
\begin{tabular}{@{}c@{\quad\quad\quad\quad\quad\quad\quad\quad\quad}c@{}}
\raisebox{5pt}{
  \begin{picture}(95,100)
   \put(0,0){\scalebox{2.5}{\includegraphics{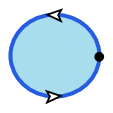}}}
   \put(0,0){
     \setlength{\unitlength}{.75pt}\put(0,0){
     \put(96,45)   { $(\EC,x_1\otimes \cdots \otimes x_n)$}
     \put(42,97)     {$ \EC $}
     \put(45,-5)     {$ \EC $}
     \put(35,45)      {$Z(\EC)$}
          }\setlength{\unitlength}{1pt}}
  \end{picture}}
  &
\raisebox{33pt}{
  \begin{picture}(95,100)
   \put(0,0){\scalebox{2.5}{\includegraphics{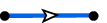}}}
   \put(0,0){
     \setlength{\unitlength}{.75pt}\put(0,0){
     \put(85,-7)   { $(\EC,x_1\otimes \cdots\otimes x_n)$}
     \put(20,22)     {$ \fun_\bk(\EC,\EC) $}
     \put(-30,-7)      {$(\EC,\one)$}
          }\setlength{\unitlength}{1pt}}
  \end{picture}}
  \\
  (a) & (b)
  \end{tabular}}
\caption{
These two pictures depicts the two steps in computing the factorization homology of a sphere with the coefficient system determined by a single UMTC $\EC$.
}
\label{fig:2sphere}
\end{figure}

\begin{thm} \label{thm:genus-0}
Let $\EC$ be a UMTC and $x_1, \dots, x_n \in \EC$. Consider the stratified sphere $S^2$ without 1-stratum but with finitely many 0-cells $p_1, \cdots, p_n$. Suppose a coefficient system assigns $\EC$ to the unique 2-cell and assigns $(\EC, x_1), \dots, (\EC, x_n)$ to the 0-cells $p_1, \dots, p_n$, respectively. We have
\be \label{eq:g=0}
\int_{(S^2; \emptyset; p_1, \cdots, p_n)} (\EC; \emptyset; (\EC,x_1), \dots, (\EC, x_n)) \simeq (\bk, \hom_\EC(\one_\EC, x_1\otimes \dots \otimes x_n)).
\ee
\end{thm}
\pf
First we project the stratified sphere to a closed stratified 2-disk as depicted in Fig.\,\ref{fig:2sphere} (a) such that the projection restricts to a 1-,2-fold covering over the 1-,2-cell, respectively, and maps all the points $p_1, \dots, p_n$ to the 0-cell. Applying the pushforward property and the $\otimes$-excision property, we reduce the problem to the computation of the factorization homology of the stratified 2-disk depicted in Fig.\,\ref{fig:2sphere} (a). Note that $\overline{\EC}\boxtimes\EC \simeq Z(\EC)$.

Then we apply the $\otimes$-excision property to complete the computation. That is, we further project the stratified 2-disk vertically onto the closed interval $[-1,1]$ as depicted in Fig.\,\ref{fig:2sphere} (b). Note that $\EC\boxtimes_{Z(\EC)}\EC^\rev \simeq \fun_\bk(\EC,\EC)$ and the induced left action of $\fun_\bk(\EC,\EC)$ on $\EC$ is the obvious one by Cor.\,\ref{cor:tensor-fun}. Moreover, since the duality functor $\delta: \fun_\bk(\EC,\EC)^\rev \to \fun_\bk(\EC,\EC)^\op$ carries $a\otimes-\otimes b$ to $b^\ast\otimes-\otimes a^\ast$ for $a,b\in\EC$, the induced right action of $\fun_\bk(\EC,\EC)$ on $\EC$ coincides with the right action of $\fun_\bk(\EC,\EC)$ on $\EC^\op$ if we identify $\EC$ with $\EC^\op$ via duality.

The final result is expressed as a tensor product:
$$
\int_{(S^2; \emptyset; p_1, \cdots, p_n)} (\EC; \emptyset;\, (\EC,x_1), \dots, (\EC, x_n)) \simeq \left(\EC \boxtimes_{\fun_\bk(\EC,\EC)} \EC, \,\,\, \one_\EC\boxtimes_{\fun_\bk(\EC,\EC)} (x_1\otimes \cdots \otimes x_n)\right).
$$
Notice that $\EC$ is an invertible $\bk$-$\fun_\bk(\EC,\EC)^\rev$-bimodule (see Rem.\,\ref{rem:inv-bimod}). We have canonical equivalences $\EC \boxtimes_{\fun_\bk(\EC,\EC)} \EC \simeq \EC^\op \boxtimes_{\fun_\bk(\EC,\EC)} \EC \simeq \EC \boxtimes_{\fun_\bk(\EC,\EC)^\rev} \EC^\op \simeq \bk$, where the first equivalence is given by the duality $\EC\simeq\EC^\op$ and the last one is defined in Eq.\,(\ref{eq:v}). The composed equivalence maps as follows:
\be \label{eq:2-sphere-pf}
x\boxtimes_{\fun_\bk(\EC,\EC)} y \mapsto [y,x^\ast]_\bk^\ast\simeq \hom_\EC(y,x^\ast)^\ast \simeq \hom_\EC(x^\ast,y).
\ee
Plugging in $x=\one_\EC$ and $y=x_1\otimes \cdots \otimes x_n$ in Eq.\,(\ref{eq:2-sphere-pf}), we obtain Eq.\,(\ref{eq:g=0}).
\epf

\void{
\begin{rem}
It is certainly not necessary to fuse all $x_1, \cdots, x_n$ to the right point in Fig.\,\ref{fig:2sphere}. If we fuse some of them to the left point, we need first use $\delta: \EC \to \EC^\op$ to map them to $\EC^\op$, then apply Eq.\,(\ref{eq:2-sphere-pf}) and the result will be the same.
\end{rem}
}

\begin{thm} \label{thm:high-genus}
Let $\EC$ be a UMTC and $x_1, \dots, x_n \in \EC$. Let $\Sigma_g$ be a closed stratified surface of genus $g$ without 1-stratum but with finitely many 0-cells $p_1, \cdots, p_n$. Suppose a coefficient system assigns $\EC$ to the unique 2-cell and assigns $(\EC, x_1), \dots, (\EC, x_n)$ to the 0-cells $p_1, \dots, p_n$, respectively. We have
\be \label{eq:high-genus}
\int_{(\Sigma_g;\emptyset;p_1,\cdots, p_n)} (\EC;\emptyset; (\EC,x_1), \cdots, (\EC,x_n)) \simeq \left( \bk, \hom_\EC(\one_\EC, x_1 \otimes \cdots \otimes x_n \otimes (\oplus_{i\in \mO(\EC)} i\otimes i^\ast)^{\otimes g} \right)
\ee
\end{thm}
\pf
Notice that Eq.\,(\ref{eq:high-genus}) holds for $g=0$ by Thm.\,\ref{thm:genus-0}. Now we assume $g>0$.
The proof of Thm.\,\ref{thm:genus-0} implies that $\int_{S^1\times \Rb} \EC\simeq \fun_\bk(\EC, \EC)$. By Eq.\,(\ref{eq:xRy}) and Rem.\,\ref{rem:tensor-bimod}, we have $\fun_\bk(\EC, \EC) \simeq \EC^\op\boxtimes\EC$ as $\fun_\bk(\EC, \EC)^\rev$-$\fun_\bk(\EC, \EC)^\rev$-bimodules, under which $\id_\EC \simeq \oplus_{i\in\mO(\EC)}\hom_\EC(-,i)^\ast\otimes i$ is mapped to $\oplus_{i\in\mO(\EC)} i\boxtimes i$.
\void{
an object $X\in \EC\boxtimes \EC^\op$. We have
\begin{align}
\hom_{\EC\boxtimes\EC^\op}(c\boxtimes d, X) &\simeq \hom_{\fun_\bk(\EC,\EC)}(\hom_\EC(-,d)^\ast \odot c, \id_\EC) \nn
&\simeq \oplus_{i\in \mO(\EC)} \hom_\EC\left( \hom_\EC(i,d)^\ast \odot c, i \right) \nn
&\simeq \oplus_{i\in \mO(\EC)} \hom_{\EC^\op}(c,i) \otimes_\Cb \hom_\EC(d,i),
\end{align}
where the second $\simeq$ is due to the semi-simpleness of $\EC$. This implies $X\simeq \oplus_i i\boxtimes i$ as objects. 
}
Therefore, we have $\int_{S^1\times \Rb} \EC\simeq (\EC\boxtimes \EC^\op, \oplus_i i\boxtimes i)$ as $\fun_\bk(\EC, \EC)$-$\fun_\bk(\EC, \EC)$-bimodules. As a consequence, when we compute factorization homology, we can replace a cylinder $S^1\times \Rb$ by two open 2-disks each equipped with a 0-cell, both of which are simultaneously labeled by $(\EC\boxtimes \EC,\oplus_i i \boxtimes i^\ast)$ (recall Rem.\,\ref{rem:0-cell-label}). Notice that we have applied $\delta: \EC^\op \to \EC$ on the second factor. In this way, we reduced the genus by one. By induction, we obtain Eq.\,(\ref{eq:high-genus}) immediately.
\epf

\begin{rem} \label{rem:benzvi}
This remark is due to David Ben-Zvi. Actually, Lurie has shown that an $E_k$ algebra in any symmetric monoidal higher category is $k$-dualizable \cite[Claim\,4.1.14]{lurie-tft} and defines a framed $k$D TQFT. Moreover, the invariants of manifolds of dimension at most $k$ are determined by factorization homology \cite[Thm.\,4.1.24]{lurie-tft}. Freed and Teleman have proved that a modular tensor category is 4-dualizable and defines an invertible 4D TQFT \cite{ft} (see also \cite{schommer-pries}). In particular, the categories associated to surfaces are equivalent to $\bk$ \cite{ft}. Therefore, the first data $\bk$ in Eq.\,(\ref{eq:high-genus}) follows from these results. 
\end{rem}

\begin{rem}
It is not surprising that the distinguished object of $\bk$ in Eq.\,(\ref{eq:high-genus}) is nothing but the Hilbert space that is assigned to the surface $\Sigma$ in the Reshetikhin-Turaev 2+1D TQFT \cite{rt,turaev} determined by $\EC$. This result was vaguely alluded in the context of cobordism hypothesis in \cite{ft} by viewing the TQFT as a relative theory \cite{freed}. It is also known that this state space is the ground state degeneracy (GSD) of the topological order associated to $\EC$. For the explicit computation of GSD from concrete lattice models, one can consult \cite{hsw} and references therein. We also want to remark that the unitarity does not play any role here. Thm.\,\ref{thm:high-genus} holds for any modular tensor category $\EC$. For unitary braided fusion categories that are not non-degenerate, the values of the factorization homology cannot be $\bk$ (see examples in \cite{bbj1}) because the associated 3d bulks are non-trivial. 
\end{rem}

\subsection{Closed stratified surfaces with 1-stratum}

\begin{lem} \label{lem:stra-af}
For any stratified surface $\Sigma$ with an anomaly-free coefficient system $A$, any one of the processes described in Example\,\ref{expl:otimes} (1),(2),(3),(4) produces a new stratified surface $\Sigma'$ and a new coefficient system $A'$. We have 
\be \label{eq:fh-pfd}
\int_{\Sigma} A \simeq \int_{\Sigma'} A'.
\ee
Moreover, the new coefficient system $A'$ is also anomaly-free. 
\end{lem}
\pf
Eq. (\ref{eq:fh-pfd}) holds by the $\otimes$-excision property of factorization homology. We need to prove that the new coefficient system $A'$ is anomaly-free. This is obvious for the processes described in Example\,\ref{expl:otimes} (1),(3),(4).

It remains to show that the target label $\EP\boxtimes_\EL\EQ$ in Fig.\,\ref{fig:trees} (b) is anomaly-free.
Let $\EM = \EM_1 \boxtimes_{\EC_1} \cdots \boxtimes_{\EC_{m-1}} \EM_m$ and $\EN = \EN_1 \boxtimes_{\ED_1} \cdots \boxtimes_{\ED_{n-1}} \EN_n$. We have $Z(\EM) \simeq Z(\EL) \simeq Z(\EN^\rev) \simeq \overline{\EC}_0 \boxtimes \ED_0$ by Thm.\,\ref{thm:closed}. Moreover, $\EM \simeq \fun_{\EL^\rev}(\EP,\EP)$ and $\EN \simeq \fun_\EL(\EQ,\EQ)$ by Cor.\,\ref{cor:tensor-fun}. Then $\EM \boxtimes_{\overline{\EC}_0 \boxtimes \ED_0} \EN \simeq \fun_\bk(\EP\boxtimes_\EL\EQ,\EP\boxtimes_\EL\EQ)$ by Eq.\,(\ref{eq:MMNN}), as desired.
\epf

\begin{thm} \label{thm:S2+defects}
Given any stratified sphere $\Sigma=(S^2;\Gamma)$ and an anomaly-free coefficient system $A$ on $\Sigma$, we have
$\int_{\Sigma} A = (\bk, u_\Sigma)$, where $u_\Sigma$ is an object in $\bk$.
\end{thm}
\pf
Applying the processes described in Example\,\ref{expl:otimes} (2),(3),(4) repeatedly, we can always reduce the graph $\Gamma$ to finitely many points on $S^2$ because all loops on $S^2$ are contractible. Then the result follows from Thm.\,\ref{thm:genus-0}.
\epf

\begin{figure}[bt]
\centerline{
\begin{tabular}{@{}c@{\quad\quad\quad\quad\quad}c}
\raisebox{-0pt}{
  \begin{picture}(150,90)
   \put(0,10){\scalebox{2.5}{\includegraphics{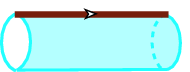}}}
   \put(0,10){
     \setlength{\unitlength}{.75pt}\put(-8,-9){
     \put(85, 76)     {$ \EM $}
     \put(48, 27)     {$ \EC $}
          }\setlength{\unitlength}{1pt}}
  \end{picture}}
&  
\raisebox{-0pt}{
  \begin{picture}(150,90)
   \put(0,10){\scalebox{2.5}{\includegraphics{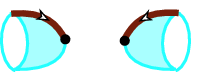}}}
   \put(0,10){
     \setlength{\unitlength}{.75pt}\put(-8,-9){
     \put(53, 70)     {$ \EM $}
     \put(40,30)      {$ \EC $}
     \put(132,72)      {$\EM^\rev$}
     \put(148,30)        {$\EC$}
     \put(78,41)      {$\EX\boxtimes\EX^\op$}
          }\setlength{\unitlength}{1pt}}
  \end{picture}}  \\
  (a) & (b)
\end{tabular}}
\caption{Figure (a) shows a stratified cylinder with a unique 2-cell labeled by a UMTC $\EC$ and a unique 1-cell labeled by a closed multi-fusion $\EC$-$\EC$-bimodule $\EM$.
Figure (b) is a disjoint union of two open 2-disks with 2-cells labeled by $\EC$, 1-cells labeled by $\EM$ and $\EM^\rev$, respectively, and two 0-cells simultaneously labeled by $(\EX\boxtimes \EX^\op, \oplus_{i\in \mO(\EX)} i\boxtimes i)$ (see Rem.\,\ref{rem:0-cell-label}).
}
\label{fig:cylinder}
\end{figure}

Before considering high genus surfaces, we prove a lemma.
\begin{lem} \label{lem:M-cylinder}
Let $\Sigma$ be the stratified cylinder $(S^1\times\Rb;\Rb)$ depicted in Fig.\,\ref{fig:cylinder} (a), in which the target label $\EC$ is a UMTC and the target label $\EM$ is a closed multi-fusion $\EC$-$\EC$-bimodule. We have
\be
\int_{(S^1\times\Rb;\Rb)} (\EC; \EM;\emptyset) \simeq \fun_\bk(\EX,\EX), 
\ee
where $\EX$ is the unique (up to equivalence) left $\EC$-module such that $\EM \simeq \fun_\EC(\EX,\EX)$. 
\end{lem}
\pf
Since $Z(\EM^\rev) \simeq \overline{\EC} \boxtimes \EC \simeq Z(\EC)$, there is a unique left $\EC$-module $\EX$ such that $\EM \simeq \fun_\EC(\EX,\EX)$ by Cor.\,\ref{cor:tensor-fun}. By the $\otimes$-excision property and Cor.\,\ref{cor:tensor-fun} again, we have $\int_{(S^1\times\Rb;\Rb)} (\EC; \EM;\emptyset)\simeq \EC\boxtimes_{Z(\EC)} \EM \simeq \fun_\bk(\EX,\EX)$, which maps $\one_\EC \boxtimes_{Z(\EC)} \one_\EM$ to $\id_\EX$. 
\epf

\begin{rem} \label{rem:red-genus}
We would like to forget the 1-disk algebra structure on $\fun_\bk(\EX,\EX)$ to obtain a 0-disk algebra $(\fun_\bk(\EX,\EX), \id_\EX)$. There is 
a canonical $\fun_\bk(\EX,\EX)$-$\fun_\bk(\EX,\EX)$-bimodule equivalence $\fun_\bk(\EX,\EX) \simeq \EX \boxtimes \EX^\op$, which maps $\id_\EX$ to $\oplus_{i\in \mO(\EX)} i\boxtimes i$ (see the proof of Thm.\,\ref{thm:high-genus}).
This suggests a very important tool of computing the factorization homology of a surface $\Sigma$ with high genus. More precisely, we can always replace any cylinder-like region in $\Sigma$ as depicted in Fig.\,\ref{fig:cylinder} (a) by two open 2-disks as depicted in Fig.\,\ref{fig:cylinder} (b) 
without changing the value of factorization homology. 
\end{rem}


\begin{thm} \label{thm:higher-g-defects}
Given any closed stratified surface $\Sigma$ and an anomaly-free coefficient system $A$ on $\Sigma$, we have
$\int_{\Sigma} A \simeq (\bk, u_\Sigma)$, where $u_\Sigma$ is an object in $\bk$.
\end{thm}
\pf
The result is true if $\Sigma$ is a stratified sphere by Thm.\,\ref{thm:S2+defects}. We assume that $\Sigma$ is of genus greater than zero. After we apply the processes described by Example\,\ref{expl:otimes} (2),(3),(4) to the graph $\Gamma$ in $\Sigma$ until no further reduction is possible, there is at least one cylinder-like region in $\Sigma$ which looks like the one depicted in Fig.\,\ref{fig:cylinder} (a). Using Lem.\,\ref{lem:M-cylinder} and Rem.\,\ref{rem:red-genus}, we can reduce the problem to the genus zero case.
\epf

\begin{rem}
Given a stratified surface with boundary $\Sigma$, one can obtain a stratified surface without boundary $\Sigma'$ by attaching 2-disks to the boundary. Therefore, the factorization homology of a stratified surface with boundary can be viewed as a special case of that of a stratified surface without boundary but with some contractible 2-cells labeled by $\bk$. We give an example in Cor.\,\ref{cor:pants}. 
\end{rem}

\void{
\begin{rem}
If the coefficient system is not anomaly-free, the value of the factorization homology might not be $\bk$ because, from the perspective of topological order, there is a non-trivial 3d bulk (or trivial 3d bulk with non-trivial defects of codimension 1 or 2). Computing factorization homology in these anomalous cases is also possible (see \cite{bbj2}).  
\end{rem}
}

\section{Topological orders} \label{sec:to}
In this section, we discuss the relation between the theory of factorization homology and 2d topological orders.

\subsection{Topological orders and anomaly-free defects} \label{sec:af}
In this subsection, we discuss topological orders and the notion of anomaly-free defects of codimension 0,1,2. In particular, we show that giving a stratified surface decorated by anomaly-free defects is equivalent to giving the same stratified surface with an anomaly-free coefficient system. All domain walls and defects in a 2d topological order are assumed to be gapped without further announcement.

\medskip
It is known that the topological excitations of an anomaly-free 2d topological order on an open 2-disk form a UMTC $\EC$. Microscopically, by ``anomaly-free" we mean that the topological order can be realized by a 2d lattice model \cite{kong-wen}. Macroscopically, by ``anomaly-free" we mean that all topological excitations are detectable by the braiding among themselves \cite{kong-wen}. In particular, it implies that if, by moving a particle-like topological excitation $x$ around another $y$ along a circular path, if we detect no physical difference for all $x$, then the excitation $y$ must be the vacuum. Namely, the only simple transparent object in $\EC$ is the tensor unit $\one_\EC$. This anomaly-free condition corresponds exactly to the non-degeneracy of the braidings in a UMTC. The simplest example of UMTC is $\bk$, which describes the trivial 2d topological order.

\begin{rem}
An anomalous 2d topological order can only be realized as a boundary of a 3d lattice model.
\end{rem}

\begin{figure}
\centerline{
\begin{tabular}{@{}c@{\quad\quad\quad\quad\quad\quad}c@{}}
\begin{picture}(100, 60)
   \put(25,0){\scalebox{2}{\includegraphics{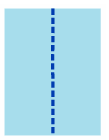}}}
   \put(25,0){
     \setlength{\unitlength}{.75pt}\put(-18,-40){
     \put(50, 148)     { \scriptsize $ \EC $}
     \put(28, 100)     { \scriptsize $\EC$}
     \put(77, 100)     { \scriptsize $\EC$}
     }\setlength{\unitlength}{1pt}}
  \end{picture} &
  \begin{picture}(100, 60)
   \put(25,0){\scalebox{2}{\includegraphics{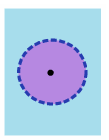}}}
   \put(25,0){
     \setlength{\unitlength}{.75pt}\put(-18,-40){
     \put(25, 130)     { \scriptsize $\EC$}
     \put(51, 110)     { \scriptsize $\EC$}
     \put(54, 87)     { \scriptsize $x$}
     }\setlength{\unitlength}{1pt}}
  \end{picture} \\
(a) & (b)
\end{tabular}}
\caption{ (a) Restricting an anomaly-free 2d topological order $\EC$ to a neighborhood of the dotted line gives a 1d topological order; (b) restricting to a neighborhood of a particle-like excitation $x\in \EC$ gives a 0d topological order described by a pair $(\EC,x)$.
} \label{fig:dr}
\end{figure}

An anomaly-free 2d topological order naturally induces topological orders of lower dimensions as illustrated in Fig.\,\ref{fig:dr}. In Fig.\,\ref{fig:dr} (a), a 2d anomaly-free topological order $\EC$ restricted to an open neighborhood of a 1-codimensional submanifold gives a 1d topological order. It is described by the unitary fusion category $\EC$ obtained from the UMTC $\EC$ by forgetting its braiding structure. A unitary fusion category is a 1-disk algebra in $\EV_{\mathrm{uty}}$ (recall Example\,\ref{expl:disk-alg}), which means that topological excitations can be fused in the oriented 1-dimensional manner. In Fig.\,\ref{fig:dr} (b), by restricting to an open neighborhood of a single topological excitation $x\in \EC$, we obtain a 0-disk algebra $(\EC,x)$, where $\EC$ is viewed as a unitary category by forgetting its monoidal structure.

\medskip
More generally, a potentially anomalous 1d topological order is described by a unitary multi-fusion category (UMFC) $\EM$ \cite{kong-wen-zheng}. It is anomaly-free as a 1d topological order if its Drinfeld center $Z(\EM)$ is trivial, i.e. $Z(\EM)\simeq \bk$. It is known that such a UMFC has a unique indecomposable left module $\EP$ (a unitary category) such that the module structure on $\EP$ gives a monoidal equivalence $\EM \simeq \fun_\bk(\EP,\EP)$. The phase $\EM$ is called {\it 2-stable} (i.e. stable as a 1d topological order) if and only if $\EM\simeq \bk$ \cite{kong-wen-zheng}. When $Z(\EM)\simeq \bk$ and $\EM\not{\simeq} \bk$, the anomaly-free 1d phase is not 2-stable, but it might be stable as a 2d phase. We give explicit lattice models to illustrate this phenomenon in Example\,\ref{expl:toric-1}, \ref{expl:toric-2}, \ref{expl:lw-mod}.

Similar to how we define the notion of an anomaly-free 2d topological order, we can define the notion of an anomaly-free domain wall (or an anomaly-free 1-codimensional defect) between two anomaly-free 2d topological orders. More precisely, in the stratified 2-disk depicted in Fig.\,\ref{fig:anomaly-free} (a), there is a domain wall $\EM$ between two 2d phases $\EC$ and $\ED$. The arrow on the wall indicates an orientation, which tell us in what order we should fuse excitations on the wall. Bulk excitations from two sides of the wall can be fused into the wall. This gives the bimodule action on $\EM$ from the UMTC's associated to the adjacent 2-cells. The arrow on the wall also tell us which one of $\EC$ and $\ED$ act on $\EM$ from left (or right). Our convention is that when you are standing on the wall and seeing the arrow coming towards you, what sits on your left (or right) acts from left (or right). For example, in Fig.\,\ref{fig:anomaly-free} (a), $\EM$ is a multi-fusion $\EC$-$\ED$-bimodule, and $\EN$ is a multi-fusion $\EE$-$\ED$-bimodule. There is no difference if we flip the arrow on the wall between $\ED$ and $\EE$ in Fig.\,\ref{fig:anomaly-free} (a) and replace $\EN$ by $\EN^\rev$.

The wall $\EM$ is called {\it anomaly-free} if one can realize it as a boundary in a 2d lattice model with the bulk phase given by $\overline{\EC}\boxtimes \ED$. By the boundary-bulk relation \cite{anyon,kong-wen-zheng}, macroscopically, the domain wall $\EM$ between $\EC$ and $\ED$ is anomaly-free if $\EM$ is a closed multi-fusion $\EC$-$\ED$-bimodule. In general, anomaly-free walls between two arbitrary UMTC's $\EC$ and $\ED$ might not exist. The UMTC's $\EC$ and $\ED$ are called Witt equivalent if there are solutions of $\EK$ to the braided eqivalence $Z(\EK)\simeq \overline{\EC}\boxtimes \ED$ \cite{dmno}. In this case, the solutions of $\EK$ are often not unique. It is known that fusion-category solutions one-to-one correspond to the Lagrangian algebras in $\overline{\EC}\boxtimes \ED$ \cite{dmno}. If $\EK$ is allowed to be multi-fusion, there are more solutions. But all solutions are Morita equivalent \cite{kong-zheng}.

\begin{rem}
An anomalous 1-codimensional defect in a 2d topological order can only be realized as a 2-codimensional defect in a 3d lattice model with a non-trivial 3d bulk phase.
\end{rem}

\begin{figure}
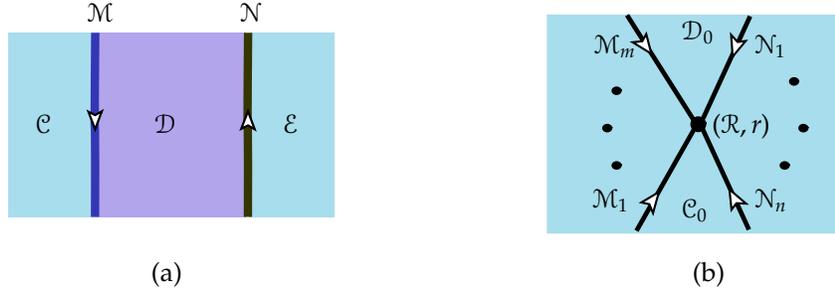

\centerline{
\begin{tabular}{@{}c@{\quad\quad\quad\quad\quad\quad}c@{}}
\begin{picture}(140, 95)
   \put(10,10){\scalebox{2}{\includegraphics{pic-ACB-2.eps}}}
   \put(-30,-54){
     \setlength{\unitlength}{.75pt}\put(-18,-19){
     \put(85, 150)    {$\EC$}
     \put(145, 150)    {$\ED$}
     \put(210, 150)    {$\EE$}
     \put(108,206)     { $\EM$}
     \put(185,206)     { $\EN$}
   }\setlength{\unitlength}{1pt}}
  \end{picture}
&
  \begin{picture}(150,100)
   \put(10,0){\scalebox{2.2}{\includegraphics{pic-trees-2.eps}}}
   \put(10,0){
     \setlength{\unitlength}{.75pt}\put(-8,-9){
     \put(93,68)   { $(\ER,r)$}
     \put(118, 108)     {$ \EN_1 $}
     \put(37, 108)     {$ \EM_m $}
     \put(36,30)      {$\EM_1$}
     \put(118,30)      {$\EN_n$}
     \put(81,115)      {$\ED_0$}
     \put(81,25)      {$\EC_0$}
          }\setlength{\unitlength}{1pt}}
  \end{picture}  \\
  (a) & (b)
\end{tabular}}
\caption{Figure (a) depicts a stratified 2-disk with three 2-cells, each of which is a 2d anomaly-free topological order, and two 1-cells, each of which is an anomaly-free domain wall. Figure (b) depicts a standard stratified 2-disk $(\Rb^2;C(I))$. The 0-cell at original point is labeled by an anomaly-free 2-codimensional defect $(\ER,r)$. }
\label{fig:anomaly-free}
\end{figure}

Let $\EC_0, \cdots, \EC_m$ and $\ED_0,\cdots, \ED_n$ be UMTC's and $\EC_m=\ED_0$ and $\EC_0=\ED_n$. Let $\EM_i$ be a closed multi-fusion $\EC_{i-1}$-$\EC_i$-bimodule for $i=1,\cdots, m$, and let $\EN_j$ be a closed multi-fusion $\ED_{j-1}$-$\ED_j$-bimodule for $j=1,\cdots, n$. Consider a 2-codimensional defect $(\ER,r)$ depicted in Fig.\,\ref{fig:anomaly-free} (b). The 2-codimensional defect $(\ER,r)$ is called anomaly-free if it can be realized as a 2-codimensional defect on the boundary of a 2d lattice model with the bulk given by $\overline{\EC_0}\boxtimes \ED_0$, the boundary on the one side of $(\ER,r)$ is given by $\EM_1\boxtimes_{\EC_1} \cdots \boxtimes_{\EC_{m-1}} \EM_m$, and that on the other side is given by $\EN_1 \boxtimes_{\ED_1} \cdots \boxtimes_{\ED_{n-1}} \EN_n$. According to \cite{kk,anyon}, all of $\overline{\EC_0}\boxtimes \ED_0$, $\EM_1\boxtimes_{\EC_1} \cdots \boxtimes_{\EC_{m-1}} \EM_m$, $\EN_1 \boxtimes_{\ED_1} \cdots \boxtimes_{\ED_{n-1}} \EN_n$ can be realized by Levin-Wen type of lattice models (see Example\,\ref{expl:lw-mod}). As a consequence, the unitary category $\ER$ is uniquely (up to equivalence) determined by other data. Mathematically, $(\ER,r)$ is anomaly-free if and only if the $\cboxtimes_{\EC_0, \cdots, \EC_{m-1}, \ED_0, \cdots, \ED_{n-1}}(\EM_1, \cdots, \EM_m, \EN_1, \cdots, \EN_n)$-module structure on $\ER$ is closed (recall Def.\,\ref{def:closed-module}), i.e.
\be \label{eq:af-cond}
\cboxtimes_{\EC_0, \cdots, \EC_{m-1}, \ED_0, \cdots, \ED_{n-1}}(\EM_1, \cdots, \EM_m, \EN_1, \cdots, \EN_n) \simeq \fun_\bk(\ER,\ER).
\ee
According to \cite{kk} (see also Example\,\ref{expl:1d-fhomology}), $(\ER,r)$ satisfying Eq.\,(\ref{eq:af-cond}) is the only type of 2-codimensional defects that are realizable in Levin-Wen type of lattice models.
Notice that we have made an ad hoc choice of partition of the set $\{ \EM_1, \cdots, \EM_m, \EN_1, \cdots, \EN_n\}$ when we define the anomaly-free condition via lattice models. But the mathematical anomaly-free condition given in Eq.\,(\ref{eq:af-cond}) is independent of the partition.

\begin{rem} 
It is worthwhile to study a special case, in which $m=1$, $n=0$. Let $\EM_1=\EM$ and $\EC_0=\EC$. In this case, $\EM^\rev$ and $\EC$ share the same Drinfeld center, consequently there is a unique left $\EC$-module $\EP$ such that $\EM\simeq \fun_\EC(\EP,\EP)$ by Cor.\,\ref{cor:tensor-fun}. The ``moreover" part of Cor.\,\ref{cor:tensor-fun} says that $\fun_\bk(\EP,\EP)\simeq\fun_\bk(\ER,\ER)$. Namely, $\ER\simeq\EP$ as $\EM$-$\EC^\rev$-bimodules. 
Now we flip the arrow on the wall and keep the label $\EM$. Then $\EM^\rev,\EC$ still share the same Drinfeld center and $\EM\simeq \fun_\EC(\EP,\EP)$ with the same $\EP$ but $\fun_\bk(\EP,\EP) \simeq \fun_\bk(\ER,\ER)^\rev \simeq \fun_\bk(\ER,\ER)^\op \simeq \fun_\bk(\ER^\op,\ER^\op)$. Therefore, $\ER\simeq\EP^\op$ as $\EC^\rev$-$\EM$-modules in the flipped case.
\end{rem}

\begin{rem} \label{rem:flip-arrow}
A special case of Fig.\,\ref{fig:anomaly-free} (b) is worth of a separate remark. If $\EM_i=\EN_j=\EC_k=\EC_j=\EC$, where $\EC$ is a UMTC, we must have $\ER\simeq\EC$. The case $\EC=\bk$ says that an anomaly-free 2-codimensional defect in the trivial 2d topological order, i.e. 0d anomaly-free topological order (see \cite[Example\,2.2.3 (1),\, Remark\,2.2.4]{kong-wen-zheng}), is given by a pair $(\bk, u)$, where $u$ is an object in $\bk$. 
\end{rem}

\begin{rem}
If the condition Eq.\,(\ref{eq:af-cond}) does not hold, then the defect $(\ER,r)$ is called an anomalous 2-codimensional defect. Such a defect can only be realized as a 3-codimensional defect in a 3d lattice model with a non-trivial 3d bulk phase.
\end{rem}

\subsection{Dimensional reduction and factorization homology} \label{sec:dr-fh}




In this subsection, we show that the following three basic dimensional reduction processes of topological orders coincide with the $\otimes$-excision property of factorization homology:
\bnu
\item stacking an anomaly-free 2d topological order on the top of another, viewed as the fusion of two such topological orders into one; 
\item the fusion of two anomaly-free 1-codimensional defects (connected by an anomaly-free 2d topological order) into one (see the picture in Eq.\,(\ref{eq:MDN})); 
\item the fusion of two anomaly-free 2-codimensional defects (connected by an anomaly-free 1-codimensional defect) into one (see Fig.\,\ref{fig:1d-fhomology}).
\enu
We study these three processes in order.

\medskip
{\bf Process 1}: If we stack two layers of anomaly-free 2d topological orders $\EC$ and $\ED$ without introducing any coupling between the two layers, topological excitations in this two-layer system form a UMTC: $\EC\boxtimes \ED$. This agrees with the following special case of the $\otimes$-excision property: if we label two copies of $\Rb^2$ by $\EC$ and $\ED$, respectively, then $\int_{\Rb^2\cup\Rb^2}(\EC,\ED;\emptyset;\emptyset) \simeq \EC\boxtimes\ED$.

\medskip
{\bf Process 2}: Consider the configuration depicted on the right hand side of Eq.\,(\ref{eq:MDN}). There are three anomaly-free 2d topological orders labeled by $\EC,\ED,\EE$. They are separated by two anomaly-free domain walls labeled by $\EM$ and $\EN$. In other words, $\EM$ is a closed multi-fusion $\EC$-$\ED$-bimodule, and $\EN$ is a closed multi-fusion $\ED$-$\EE$-bimodule.

On the one hand, according to the $\otimes$-excision property, the factorization homology of the whole region depicted on the right hand side of Eq.\,(\ref{eq:MDN}) should be given by $\EM\boxtimes_\ED \EN$. If we map the whole region to a vertical line (parallel to the two walls in Eq.\,(\ref{eq:MDN})), then by the pushforward property, we obtain a 1-disk algebra (or a 1d topological order) on the line given by $\EM\boxtimes_\ED \EN$.

On the other hand, viewed from far away, the walls $\EM$ and $\EN$ fuse into a new wall, which is denoted by $\EM\times_\ED \EN$, i.e.
\be \label{eq:MDN}
\EM \times_{\ED} \EN :=
\raisebox{-3.5em}{ \begin{picture}(140, 75)
   \put(10,10){\scalebox{1.5}{\includegraphics{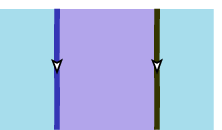}}}
   \put(-30,-54){
     \setlength{\unitlength}{.75pt}\put(-18,-19){
     \put(85, 140)    {$\EC$}
     \put(125, 140)    {$\ED$}
     \put(180, 140)    {$\EE$}
     \put(93,182)     { $\EM$}
     \put(155,182)     { $\EN$}
   }\setlength{\unitlength}{1pt}}
  \end{picture}}\,.
\ee
Suppose $\EC, \ED, \EE$ are all non-chiral. Namely, they are all Drinfeld centers of some unitary fusion categories. By the constructions of Levin-Wen type of models with domain wall given in \cite{kk}, one can realize $\EM \times_{\ED} \EN$ as $\EM \boxtimes_{\ED} \EN$ in these lattice models. This is also consistent with and a consequence of Example\,\ref{expl:1d-fhomology} via the boundary-bulk duality \cite{kong-wen-zheng, kong-zheng}. For chiral cases, by applying the folding trick (folding along an arbitrary vertical line in the $\ED$-phase), one can reduce the problem to non-chiral cases. Such fusion of domain walls $\EM$ and $\EN$ has already been proposed and studied in \cite{fs}. Moreover, the fused wall $\EM\boxtimes_\ED\EN$ is automatically anomaly-free by Thm.\,\ref{thm:closed} (see also \cite{kawahigashi}). 
\void{
One can also prove $\EM \times_{\ED} \EN\simeq\EM \boxtimes_{\ED} \EN$ model independently by the theory of anyon condensation \cite{anyon}. More precisely, $\EM\times_\ED \EN$, as a gapped boundary of the bulk phase $\overline{\EC}\boxtimes \EE$, must be obtained by an anyon condensation, which is completely determined by the functors $L: \EC \to \EM$ and $R: \EE \to \EN$. By the universal property of $\EM\boxtimes_\ED\EN$, it is clear that there is a canonical monoidal functor $F: \EM\boxtimes_\ED\EN \to \EM\times_\ED\EN$ such that $\overline{\EC}\boxtimes \EE \to \EM\times_\ED\EN$ factors through $F$. Moreover, $F$ must be surjective (containing all simple objects as the direct summands of $F(x)$ for $x\in \EM\boxtimes_\ED\EN$). Otherwise, some excitations in $\EM\times_\ED\EN$ cannot be detected by the bulk excitations in $\overline{\EC}\boxtimes \EE$ (\cite{levin,anyon}). By the surjectivity of $F$, we can show that $F^\vee(\one_{\EM\times_\ED\EN})$ must include $\one_{\EM\boxtimes_\ED\EN}$ as a subalgebra, where $F^\vee$ is the right adjoint functor of $F$. As a consequence, $(L\boxtimes R)^\vee F^\vee(\one_{\EM\times_\ED\EN})$ must include the Lagrangian algebra $(L\boxtimes R)^\vee(\one_{\EM\boxtimes_\ED\EN})$ as a subalgebra. On the other hand, $(L\boxtimes R)^\vee F^\vee(\one_{\EM\times_\ED\EN})$ must also be a Lagrangian algebra. Therefore, $F$ must be a monoidal equivalence.  
}

Therefore, we must have $\EM \times_{\ED} \EN\simeq\EM \boxtimes_{\ED} \EN$ as long as $\EC,\ED,\EE$ and $\EM,\EN$ are all anomaly-free. In other words, the dimensional reduction process corresponds precisely to the $\otimes$-excision property and the pushforward property of factorization homology.

\begin{rem}
In general, even if $\EM$ and $\EN$ are both unitary fusion categories, $\EM\boxtimes_\ED \EN$ is multi-fusion in general. So $\EM\boxtimes_\ED \EN$ is not 2-stable as 1d phase in general \cite{kong-wen-zheng}, and can flow to a stable one under local perturbations. We give an example in Example\,\ref{expl:toric-2}. But when both $\EM$ and $\EN$ are fusion, $\EM\boxtimes_\ED \EN$ is stable as a 2d phase with two 1d defects.
\end{rem}

\begin{rem}
We have shown that $\EM \times_{\ED} \EN\simeq \EM \boxtimes_{\ED} \EN$ for anomaly-free $\EC,\ED,\EE,\EM,\EN$. We believe that the equivalence still holds for anomalous $\EC,\ED,\EE,\EM,\EN$. 
\end{rem}

We will illustrate this correspondence between the dimensional reduction processes in topological orders and the theory of factorization homology in a few concrete lattice models.
Example \ref{expl:toric-1}, \ref{expl:toric-2} and \ref{expl:lw-mod} have already appeared in \cite{kong-wen-zheng}, but we put it in a new context. In particular, we emphasize the correspondence to the $\otimes$-excision property and the pushforward property of factorization homology.

\begin{expl} \label{expl:toric-1} (Toric code model, I)
Consider a narrow band of toric code model depicted in Fig.\,\ref{fig:toric} (a), in which the left side is the smooth boundary and the right side is the rough boundary \cite{bk}. In this case, note that $e$-particles condense on the rough boundary and $m$-particles condense on the smooth boundary. As a consequence, no particle survives except the trivial one. Therefore, this narrow band of toric code model realizes the trivial 1d topological order $\bk$. On the other hand, the particles on the smooth boundary form the unitary fusion category $\fun_{\rep(\Zb_2)}(\rep(\Zb_2), \rep(\Zb_2))$, those on the rough boundary form $\fun_{\rep(\Zb_2)}(\bk, \bk)$. The bulk excitations form the Drinfeld center $Z(\rep(\Zb_2))$ of $\rep(\Zb_2)$. According to the $\otimes$-excision property and the pushforward property of factorization homology, the 1d topological order obtained by squeezing the band should be given by
\begin{align}
&\fun_{\rep(\Zb_2)}(\rep(\Zb_2), \rep(\Zb_2)) \boxtimes_{Z(\rep(\Zb_2))}
\fun_{\rep(\Zb_2)}(\bk, \bk)  \nn
&\hspace{0.3cm} \simeq \fun_{\bk}(\rep(\Zb_2) \boxtimes_{\rep(\Zb_2)} \bk,\,\, \rep(\Zb_2) \boxtimes_{\rep(\Zb_2)} \bk) \nn
&\hspace{0.3cm}\simeq \fun_{\bk}(\bk,\,\, \bk) \simeq \bk,  \label{eq:toric-code-1}
\end{align}
where we have used Eq.\,(\ref{eq:MMNN}) in the first step. Therefore, the theory of factorization homology coincides with dimensional reduction in topological orders in this case.
\end{expl}

\begin{figure}[bt]
\centerline{
\begin{tabular}{@{}c@{\quad\quad\quad\quad\quad\quad}c@{}}
\raisebox{-50pt}{
  \begin{picture}(95,130)
   \put(0,0){\scalebox{4}{\includegraphics{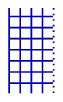}}}
    \put(50,110)     {\scriptsize $ \mbox{rough boundary} $}
  \end{picture}} &
\raisebox{-50pt}{
  \begin{picture}(75,130)
   \put(0,0){\scalebox{4}{\includegraphics{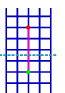}}}
   \put(0,0){
     \setlength{\unitlength}{.75pt}\put(-18,-19){
     \put(66, 130)       {\scriptsize $ \bar{e} $}
     \put(66, 50)       {\scriptsize $ e $}
     \put(5,80)     {\scriptsize $m$}
     \put(110,80)    {\scriptsize $m$}
     \put(-40,15)     {\scriptsize $ \mbox{smooth boundary} $}
     }\setlength{\unitlength}{1pt}}
  \end{picture}} \\
(a) &  (b)
\end{tabular}}
\caption{In (a), a narrow band of a toric code model bounded by two different gapped boundaries; in (b), both boundaries are the smooth boundaries.  
}
\label{fig:toric}
\end{figure}

\begin{expl} \label{expl:toric-2} (Toric code model, II)
Consider a narrow band of toric code model depicted in Fig.\,\ref{fig:toric} (b), where the both sides are the smooth boundaries. In this case, only $m$-particles are condensed, and $e$-particles survive in the 1d phase. One can create a pair of $e$-particles by the string operator illustrated in Fig.\,\ref{fig:toric} (b). But now this $e$-string is detectable by the string operator that creates a pair of $m$-particles. This $m$-string operator is local when the band is narrow. As a consequence, the vacuum of the 1d phase is two-fold degenerate \cite{hw}. Moreover, two ends of the $e$-string correspond to two different particle types because they connect two different vacuum in different ways with respect to the orientation of the 1d space manifold. We use $e$ and $\bar{e}$ to distinguish them. Therefore, excitations in this narrow band of toric model should form the following UMFC of the matrix type
\be \label{eq:matrix-hilb}
\left( \begin{array}{cc}  \bk & \bk  \\
 \bk & \bk  \end{array} \right)
\ee
with the fusion product given by the matrix multiplication.

On the other hand, according to the properties of factorization homology, the resulting 1d phase should be given by the UMFC:
\begin{align}
&\fun_{\rep(\Zb_2)}(\rep(\Zb_2), \rep(\Zb_2)) \boxtimes_{Z(\rep(\Zb_2))}
\fun_{\rep(\Zb_2)}(\rep(\Zb_2), \rep(\Zb_2)) \nn
&\hspace{0.3cm}\simeq \fun_{\bk}(\rep(\Zb_2) \boxtimes_{\rep(\Zb_2)} \rep(\Zb_2),\,\, \rep(\Zb_2) \boxtimes_{\rep(\Zb_2)} \rep(\Zb_2)) \nn
&\hspace{0.3cm}\simeq  \fun_{\bk}(\rep(\Zb_2),\,\, \rep(\Zb_2)), \label{eq:toric-code-ACB}
\end{align}
which is equivalent to Eq.\,(\ref{eq:matrix-hilb}). Therefore, the theory of factorization homology coincides with dimensional reduction in topological orders in this case. Also note that due to the vacuum degeneracy, such obtained 1d phase (Eq.\,(\ref{eq:matrix-hilb})) is not stable. It can flow to the only stable one $\bk$ under the perturbation by local operators (such as the $m$-string operators).
\end{expl}

\begin{expl} \label{expl:lw-mod} (Levin-Wen Models)
Consider a Levin-Wen type of lattice model depicted in Fig.\,\ref{fig:lw-mod} (b), the bulk lattice defined by a unitary fusion category $\EC$, the upper/lower boundary lattice is defined by an indecomposable right $\EC$-module $\EN$/$\EM$. The excitations in the bulk are given by the Drinfeld center $Z(\EC)$ of $\EC$, the excitations on the $\EM$-boundary by $\fun_{\EC^\rev}(\EM, \EM)$, those on $\EN$-boundary by $\fun_{\EC^\rev}(\EN, \EN)$ and the defect junction (the purple dot) by a $\EC$-module functor $f\in \fun_{\EC^\rev}(\EM, \EN)$. When the defect junction is viewed as a 0d topological order (by including the action of nearby excitations on $f$), it is given by the pair $(\fun_{\EC^\rev}(\EM, \EN), f)$. By folding the two boundaries along two dotted arrow, we obtain Fig.\,\ref{fig:lw-mod} (b), in which the 1d phase $\EE$ should be given by
\begin{align} \label{eq:lw-mod-dim-red}
\EE &= \fun_{\EC^\rev}(\EN, \EN) \boxtimes_{Z(\EC)} \fun_{\EC^\rev}(\EM, \EM)^\rev
\end{align}
according to the $\otimes$-excision property and the pushforward property of factorization homology. On the other hand, when we fold $\EM$-boundary upwards and flip its orientation, the right $\EC$-module $\EM$ becomes the left $\EC$-module $\EM^\op$. It amounts to a Levin-Wen type of lattice model defined by $\bk$-lattices together with a domain wall defined by $\bk$-$\bk$-bimodule $\EM^{\op}\boxtimes_\EC \EN$.
The excitations on this narrow band should be given by
\begin{align}  \label{eq:lw-mod-dim-red-2}
\fun_\bk(\EN\boxtimes_\EC \EM^{\op}, \EN\boxtimes_\EC \EM^{\op}).
\end{align}
Using Eq.\,(\ref{eq:MMNN}), it is clear that Eq.\,(\ref{eq:lw-mod-dim-red}) and Eq.\,(\ref{eq:lw-mod-dim-red-2}) are canonically equivalent. It is also interesting to note that $\EE$ is nothing but the center of the 0-disk algebra $(\fun_{\EC^\rev}(\EM, \EN),f)$ \cite{kong-wen-zheng}.
\end{expl}

\begin{figure}[bt]
\centerline{
\begin{tabular}{@{}c@{\quad\quad\quad\quad\quad}c@{}}
\raisebox{-0pt}{
  \begin{picture}(95,120)
   \put(0,0){\scalebox{1}{\includegraphics{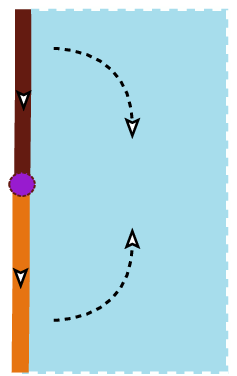}}}
   \put(0,0){
     \setlength{\unitlength}{.75pt}\put(-18,-19){
     \put(-15, 167)       {\scriptsize $ \fun_{\EC^\rev}(\EN, \EN) $}
     \put(-15, 10)     {\scriptsize $ \fun_{\EC^\rev}(\EM, \EM) $}
     \put(-55, 94)   {\scriptsize $(\fun_{\EC^\rev}(\EM, \EN), f)$}
     \put(80, 148)     {\scriptsize $ Z(\EC) $}
          }\setlength{\unitlength}{1pt}}
  \end{picture}}
  &
  \raisebox{-20pt}{
  \begin{picture}(105,120)
   \put(0,25){\scalebox{1}{\includegraphics{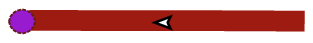}}}
   \put(0,25){
     \setlength{\unitlength}{.75pt}\put(-18,-19){
     \put(-20, 45)       {\scriptsize $ (\fun_{\EC^\rev}(\EM, \EN),f) $}
     \put(82, 69) {\scriptsize  $\EE$}
     }\setlength{\unitlength}{1pt}}
  \end{picture}} \\
(a) &  (b)
\end{tabular}}
\caption{This figure depicts a dimensional reduction process in a Levin-Wen type of lattice model discussed in Example\,\ref{expl:lw-mod}. $\EE$ is given by (\ref{eq:lw-mod-dim-red}). The arrow on the boundary of (b) indicate that we write the fusion of boundary excitations from top to bottom as a tensor product from left to right on a horizontal line.
}
\label{fig:lw-mod}
\end{figure}


\smallskip
{\bf Process 3}: Without loss of generality, we consider the configuration depicted in Fig.\,\ref{fig:1d-fhomology}, in which the bulk phase is given by $Z(\EC)$, where $\EC$ is a unitary fusion category. By Thm.\,\ref{thm:kz}, the only possible anomaly-free boundaries are those UMFC's that are Morita equivalent to $\EC$. Therefore, they are given by $\fun_{\EC}(\EK,\EK)$ for some left $\EC$-module $\EK$. In Fig.\,\ref{fig:1d-fhomology}, we give three such anomaly-free boundaries. If $\EP$ and $\EQ$ are anomaly-free 2-codimensional defects, they are uniquely determined by the anomaly-free condition Eq.\,(\ref{eq:af-cond}). Applying Eq.\,(\ref{eq:MMNN}) and Eq.\,(\ref{eq:xRy}), we obtain that $\EP=(\fun_\EC(\EL, \EM), f)$ and $\EQ=(\fun_\EC(\EM, \EN),g)$ for an arbitrary choice of $f$ and $g$. On the one hand, the $\otimes$-excision property of factorization homology predicts that the triple $(\EQ,\fun_\EC(\EM,\EM), \EP)$ should combine to a single 2-codimensional defect given by
\be \label{eq:gf}
(\EQ\boxtimes_{\fun_\EC(\EM,\EM)} \EP, g\boxtimes_{\fun_\EC(\EM,\EM)} f) \simeq
(\fun_\EC(\EL, \EN), g\circ f).
\ee
On the other hand, the dimensional reduction process should give again an anomaly-free 2-codimensional defect, which has to be $(\fun_\EC(\EL, \EN), x)$ for some object $x\in \fun_\EC(\EL, \EN)$. We show in Example\,\ref{expl:1d-fhomology} via concrete lattice models that $x$ has to be $g\circ f$.

\begin{expl} \label{expl:1d-fhomology} (Dimensional reduction from 1d to 0d)
Consider a Levin-Wen type of lattice model with gapped boundaries as depicted in Fig.\,\ref{fig:1d-fhomology}. More precisely, the lattice in the bulk is constructed from a unitary fusion category $\EC$; the lattice on the three different types of boundaries are constructed from three left $\EC$-modules $\EL,\EM,\EN$, respectively. As a consequence, the bulk excitations form a UMTC $Z(\EC)$, the boundary excitations on the three different boundaries form UMFC's $\fun_\EC(\EL,\EL)$,$\fun_\EC(\EM,\EM)$ and $\fun_\EC(\EN,\EN)$, respectively.
To determine the 2-codimensional defect between the $\EL$-boundary and the $\EM$-boundary, we first recall that the boundary conditions $\EL$ and $\EM$ uniquely determines a local operator algebra $A_{(\EM,\EL)}$ \cite{kk}, which is given by
$$
A_{\EM,\EL} := \bigoplus_{j \in \mO(\EC)}\,
\bigoplus_{\lambda,\sigma\in \mO(M)}\, \bigoplus_{\gamma,\rho\in \mO(\EL)}
\hom_\EM(j\otimes \sigma, \lambda) \otimes \hom_\EL(\gamma, j\otimes \rho)\quad \left(\text{graphically,}\quad\figbox{1.0}{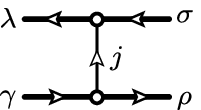}\right).
$$
It was shown in \cite{kk} that the category of all possible 2-codimensional defects between the $\EL$-boundary and the $\EM$-boundary is given by the category of $A_{(\EM,\EL)}$-modules. The later category is equivalent to the category of $\fun_\EC(\EL,\EM)$ canonically \cite{kk}. Therefore, we have $\EP=(\fun_\EC(\EL, \EM), f)$ for an arbitrary choice of object $f$. Similarly, we have $\EQ=(\fun_\EC(\EM, \EN),g)$.

\begin{figure}[bt]
$$
\raisebox{-0pt}{
  \begin{picture}(95,90)
   \put(-20,0){\scalebox{1.2}{\includegraphics{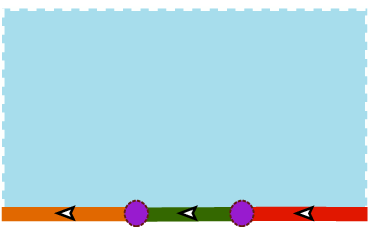}}}
   \put(-20,0){
     \setlength{\unitlength}{.75pt}\put(-18,-19){
     \put(0, 12)       {\scriptsize $ \fun_\EC(\EL, \EL) $}
     \put(160, 12)     {\scriptsize $ \fun_\EC(\EN, \EN) $}
     \put(80, 9)     {\scriptsize $ \fun_\EC(\EM, \EM) $}
     \put(78, 36)   {$\EP$}
     \put(125,36)  {$\EQ$}
     \put(90, 88)     {$ Z(\EC) $}
          }\setlength{\unitlength}{1pt}}
  \end{picture}}
$$
\caption{This figure depicts three different boundaries of an anomaly-free 2d topological order $Z(\EC)$. They are described by UMFC's $\fun_\EC(\EL, \EL)$, $\fun_\EC(\EM, \EM)$ and $\fun_\EC(\EN, \EN)$, respectively. Here $\EL,\EM,\EN$ are left $\EC$-modules.} 
\label{fig:1d-fhomology}
\end{figure}

If we squeeze the line between $\EP$ and $\EQ$ to obtain a 0d topological order, by the results in \cite{kk}, the 2-codimensional defect thus obtained must be a module over the local operator algebra $A_{\EN,\EM}$. Therefore, the category of such 2-codimensional defects must be $\fun_\EC(\EL,\EN)$. To determine the distinguished element $x$ in $\fun_\EC(\EL,\EN)$, we denoted the $A_{(\EM,\EL)}$-module associated to $f$ by $X_f$ and the
$A_{(\EN,\EM)}$-module associated to $g$ by $Y_g$. The squeezing process produces a new defect given by $Y_g\otimes_\Cb X_f$. Note that the algebra $A_{(\EN,\EL)}$ acts on the vector space $Y_g\otimes_\Cb X_f$ via a comultiplication $A_{(\EN,\EL)} \to A_{(\EN,\EM)} \otimes_\Cb A_{(\EM,\EL)}$ defined by the same formula as the one in \cite[Fig.6]{kk} (now with $\lambda_1,\lambda_2\in \EN$, $\rho,\tau\in\EM$, $\gamma_1,\gamma_2\in \EL$). It is routine to check that the object $x\in \fun_\EC(\EL,\EN)$ canonically associated to $Y_g\otimes_\Cb X_f$ is nothing but $g\circ f$. Therefore, we have shown, in this case, that the result of dimensional reduction coincides precisely with the prediction from the theory of factorization homology (see Eq.\,(\ref{eq:gf})).
\end{expl}

\begin{rem}
In this work, we restrict ourselves to the 2d systems with only anomaly-free defects of codimension 0,1,2. Using the unique bulk hypothesis proposed in \cite{kong-wen,kong-wen-zheng}, it is possible to prove the correspondence between factorization homology and the dimensional reduction processes for anomalous higher codimensional defects in higher dimensional theories.
\end{rem}


\subsection{Factorization homology and ground state degeneracy}
In this subsection, we combine results in Sec.\,\ref{sec:af}, \ref{sec:dr-fh} to argue that the following statement is true. 
\begin{quote}
Factorization homology on a closed stratified surface $\Sigma$ with an anomaly-free coefficient $A$, i.e. $\int_\Sigma A = (\bk, u_\Sigma)$, gives exactly the ground state degeneracy (GSD) of the same surface decorated by anomaly-free topological defects of codimensions 0,1,2 that are associated to $A$, i.e. $u_\Sigma = \mathrm{GSD}$. 
\end{quote} 

Using the three basic dimensional reduction processes repeatedly, together with the folding trick, we can reduce the stratifed surface $\Sigma$ decorated by anomaly-free defects associated to $A$ to a 0d topological order $X$, which is also anomaly-free because the anomaly-free condition is preserved in these processes. This 0d topological order $X$ can also be viewed as an anomaly-free 2-codimensional defect in the trivial 2d topological order. According to Remark\,\ref{rem:flip-arrow}, it can be described by a pair $(\bk, v_\Sigma)$. Since these three processes are compatible with the $\otimes$-excision property of factorization homology, we conclude that $v_\Sigma=u_\Sigma$. It remains to show that $v_\Sigma = \mathrm{GSD}$. 


Consider a 2d lattice model realization of the closed stratified surface $\Sigma$ decorated by anomaly-free defects of codimension 0,1,2 that are associated to the coefficient $A$. Such a lattice model consists of a Hilbert space $\EH_{\mathrm{tot}}$, a local structure $ \EH_{\mathrm{tot}} = \otimes_{v} \EH_v$, where $\EH_v$ is the spin space on each vertex $v$, and a Hamiltonian operator $H$. We denote the space of ground states by $\EH_{\mathrm{gs}}$. Without lose of generality, we assume $H|_{\EH_{\mathrm{gs}}}=\id_{\EH_{\mathrm{gs}}}$. Now we gradually shrink (or dimensionally reduce) the surface $\Sigma$ to a point. Then all vertices collapse to the single point. It means that the local structure is completely lost. But the total Hilbert space and the Hamiltonian remains the same. So does the space of ground states $\EH_{\mathrm{gs}}$. In this way, we obtain a lattice model realization of the 0d topological order $X$. Since a topological order only depends on the property of the ground states, picking out only $\EH_{\mathrm{gs}}$, we obtain a pair $(\EH_{\mathrm{gs}}, \id_{\EH_{\mathrm{gs}}})$. To have the complete set of observables, we should also include all the linear operators acting on $\EH$, which can be viewed as instantons or defects on the time axis \cite{kong-wen, kong-wen-zheng}. These operators form an algebra $\mathrm{End}(\EH_{\mathrm{gs}})$. Therefore, we obtain a triple $(\EH_{\mathrm{gs}}, \id_{\EH_{\mathrm{gs}}},\mathrm{End}(\EH_{\mathrm{gs}}))$. Note that this triple is exactly equivalent to data included in the pair $(\bk, \EH_{\mathrm{gs}})$. Therefore, the pair $(\bk, \EH_{\mathrm{gs}})$ also gives a characterization of the 0d topological order $X$ (see Remark\,\ref{rem:unstable}). Therefore, we must have $v_\Sigma=\EH_{\mathrm{gs}}$.

\begin{rem} \label{rem:unstable}
If $\dim_\Cb \EH_{\mathrm{gs}} >1$, the pair $(\bk, \EH_{\mathrm{gs}})$, viewed as a 0d topological order, is physically unstable because we can easily lift the GSD by introducing perturbations to the Hamiltonian $\id_\EH$. Under perturbations, we obtain the unique 0d topological order given by $(\bk, \Cb)$. However, the pair $(\bk, \EH_{\mathrm{gs}})$ for $\dim_\Cb \EH_{\mathrm{gs}} >1$ are still physically meaningful because they naturally occur as the result of the dimensional reduction of an higher dimensional topological orders. More precisely, the pair $(\bk, \EH_{\mathrm{gs}})$ is unstable as a 0d topological order when $\Sigma$ shrinks to a point, but it is stable as a 2d topological order if $\Sigma$ is slightly rescaled away from the point. This phenomenon was further explained in \cite[Example\,2.2.3 (1),\, Remark\,2.2.4]{kong-wen-zheng} and in Example\,\ref{expl:toric-2}. 
\end{rem}

\medskip


The proofs of Thm.\,\ref{thm:S2+defects} and Thm.\,\ref{thm:higher-g-defects} and Rem.\,\ref{rem:red-genus} tell us how to compute factorization homology. In the rest of this paper, we give some explicit computations of factorization homology in concrete cases (not limited to closed stratified surfaces) and compare our results on closed stratified surfaces with those on GSD in physics literature.

\subsection{Computing factorization homology: I} \label{sec:fh-comp1}

\medskip
Consider a stratified 2-disk $(\Rb^2;\Gamma_0)$ depicted in Fig.\,\ref{fig:2disk}. The closed multi-fusion $\EC$-$\ED$-bimodule structure on $\EM$ induces two monoidal functors $\EC \xrightarrow{L} \EM \xleftarrow{R} \ED$. The two target labels for 0-cells are clearly anomaly-free.

\begin{prop} \label{prop:2disk}
We have
\be \label{eq:2disk}
\int_{(\Rb^2;\Gamma_0)} \left(\EC,\ED; \EM; (\ED,d), (\EM,m)\right) \simeq \left(\EC, \,\, L^\vee(m\otimes R(d))\right),
\ee
where $L^\vee$ is the right adjoint functor of $L$.
\end{prop}
\pf
First, we fuse $(\EM,m)$ and $(\ED,d)$ which yields $(\EM,m\otimes R(d))$. Then, apply the processes depicted in Fig.\,\ref{fig:loop} and Fig.\ref{fig:trees} to contract the loop to a point. The finial result is
$$
\int_{(\Rb^2;\Gamma_0)}  \left(\EC,\ED; \EM; (\ED,d), (\EM,m)\right) 
\simeq (\EM \boxtimes_\EE \EM, \,\, \left( m\otimes R(d))\boxtimes_\EE\one_\EM \right)
$$ 
where $\EE = \EM\boxtimes_\ED \EM^\rev$. Since $\EE^\rev$ shares the same Drinfeld center as $\EC$, $\EE \simeq \fun_\EC(\EM,\EM)$ by Cor.\,\ref{cor:tensor-fun}. Moreover, $\EM$ is an invertible $\EC$-$\EE^\rev$-bimodule (see Rem.\,\ref{rem:inv-bimod}). We have $\EM \boxtimes_\EE \EM \simeq \EM^\op \boxtimes_\EE \EM \simeq \EM\boxtimes_{\EE^\rev} \EM^\op \simeq \EC$ where the last equivalence is defined by $v$ in Eq.\,(\ref{eq:v}). This composite equivalence maps as follows:
$$ \label{eq:2disk-pf}
x\boxtimes_\EE y \mapsto [y,x^\ast]_\EC^\ast \simeq [x^\ast,y]_\EC \simeq [\one_\EM, x\otimes y]_\EC \simeq L^\vee(x\otimes y),
$$
where the isomorphism $[y,x^\ast]_\EC^\ast \simeq [x^\ast,y]_\EC$ is due to the unitarity of $\EC$ and $\EM$. Setting $x=m\otimes R(d)$ and $y=\one_\EM$, we obtain Eq.\,(\ref{eq:2disk}) immediately.
\epf

\begin{figure}[bt]
$$
\raisebox{-0pt}{
  \begin{picture}(95,100)
   \put(-20,0){\scalebox{2.5}{\includegraphics{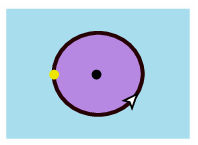}}}
   \put(-20,0){
     \setlength{\unitlength}{.75pt}\put(-8,-9){
     \put(165,120)  {$\EC$}
     \put(96,62)   { $(\ED,d)$}
     \put(105, 95)     {$ \ED $}
     \put(16,73)      {$(\EM,m)$}
     \put(140,39)      {$\EM$}
          }\setlength{\unitlength}{1pt}}
  \end{picture}}
$$
\caption{This figure depicts a stratified 2-disk with the two 2-cells labeled by UMTC's $\EC$ and $\ED$, respectively; the unique 1-cell labeled by a closed multi-fusion $\EC$-$\ED$-bimodule $\EM$; the two 0-cells labeled by $(\ED,d)$ and $(\EM,m)$, for $d\in \ED$, $m\in \EM$, respectively.}
\label{fig:2disk}
\end{figure}

\void{
\begin{rem}
Since $\EC$ is semisimple, there is an (not canonical) isomorphism $L^\vee(m^\ast)\simeq L^\vee(m)^\ast$ as objects for all $m\in \EM$. This is because we have the following identities
$$
\hom_\EC(i, L^\vee(p^\ast)) \simeq \hom_\EM(p, L(i^\ast)) \simeq 
\hom_\EM(L(i^\ast),p)^\ast \simeq \hom_\EM(L^\vee(p)^\ast, i)^\ast \simeq 
\hom_\EC(i, L^\vee(p)^\ast)
$$
for all $i\in \mO(\EC), p\in \mO(\EM)$. Therefore, Eq.\,(\ref{eq:2disk-1}) can be rewritten as
\be \label{eq:2disk}
\int_{(\Rb^2;\Gamma_0)} \left(\EC,\ED; \EM; (\ED,d), (\EM,m)\right) \simeq \left(\EC, \,\, L^\vee(m\otimes R(d))\right).
\ee
\end{rem}
}

Let $A$ be a connected commutative separable algebra in a UMTC $\EC$. Let $\EC_A$ be the category of right $A$-modules in $\EC$. A right $A$-module $M$ in $\EC$ (an object $M$ equipped with a unital right $A$-action $\mu_M: M \otimes A \to M$) is called {\it local} if $\mu_M \circ c_{A,M} \circ c_{M,A} = \mu_M$.  Let $\EC_A^0$ be the category of local right $A$-modules in $\EC$. According to the anyon condensation theory \cite{anyon}, a boson condensation of $A$ in the 2d topological phase $\EC$ produces a new phase with bulk excitations given by $\EC_A^0$ and gapped domain wall (a 1d phase) with wall excitations given by the unitary fusion category $\EC_A$. This motivates us to consider a special case of Prop\,\ref{prop:2disk}.

\begin{cor} \label{cor:2disk}
In the setting of Prop.\,\ref{prop:2disk}, if $\ED=\EC_A^0$, $\EM=\EC_A$ for a connected commutative separable algebra $A$ in $\EC$, $d=\one_\ED$ and $m=\one_\EM$, then we have
$$
\int_{(\Rb^2;\Gamma_0)} \left(\EC,\EC_A^0; \EC_A; (\EC_A^0,\one_{\EC_A^0}), (\EC_A, \one_{\EC_A})\right) \simeq (\EC, \,\, A).
$$
\end{cor}
\pf
In this case, $R: \EC_A^0 \to \EC_A$ is the canonical embedding and $L^\vee$ is the forgetful functor $\EC_A \to \EC$. Therefore, $L^\vee(\one_\EM \otimes R(\one_\ED))\simeq A$.
\epf

\begin{rem}
The closed 2-disk with the 2d phase $\EC_A^0$ in the interior and the 1d phase $\EC_A$ on its boundary behave like the object $A$ in $\EC$ was first noticed by Davide Gaiotto \cite{gaiotto}. Cor.\,\ref{cor:2disk} shows that the mathematical mechanism behind this intuition is factorization homology.
\end{rem}

If we compactify the stratified 2-disk $(\Rb^2;\Gamma_0)$ depicted in Fig.\,\ref{fig:2disk} by adding a point at infinity, we obtain a stratified sphere $(S^2;\Gamma_0)$ with the same coefficient as shown in Fig.\,\ref{fig:2disk}.
\begin{cor}
We have
$$
\int_{(S^2;\Gamma_0)} (\EC, \ED; \EM; (\ED,d), (\EM,m)) \simeq \left( \bk, \,\, \hom_\EC(\one_\EC, L^\vee(m\otimes R(d)) \right).
$$
\end{cor}
\pf
Combine Prop.\,\ref{prop:2disk} and Thm.\,\ref{thm:genus-0}.
\epf

Let $\Sigma$ be a connected compact surface whose boundary consists of $n$ circles, viewed as a stratified surface $(\Sigma,\partial\Sigma)$. Let $A$ be a coefficient system on $\Sigma$ that labels the unique 2-cell by a UMTC $\EC$ and labels the 1-cells by $\EC_{A_1}, \EC_{A_2}, \cdots,\EC_{A_n}$, where $A_i$ are Lagrangian algebras in $\EC$. A Lagrangian algebra $A$ in $\EC$ is a connected commutative separable algebra $A$ such that $\EC_A^0\simeq \bk$ \cite{dmno}. Note that a gapped boundaries of $\EC$ is given by $\EC_A$ for a unique Lagrangian algebra $A$ in $\EC$ \cite{anyon}. The factorization homology of $\Sigma$ is the same as that of a closed stratified surface $\Sigma'$ obtained by attaching 2-disks to the boundary of $\Sigma$ and with the new 2-cells labeled by $\bk$. Combining Cor.\,\ref{cor:2disk} and Thm.\,\ref{thm:high-genus}, we obtain the following result. 
\begin{cor} \label{cor:pants}
Let $g$ be the genus of the surface $\Sigma$. We have
$$
\int_\Sigma A \simeq \int_{(\Sigma';S^1\cup\cdots\cup S^1)} (\EC,\bk,\cdots,\bk; \EC_{A_1}, \cdots, \EC_{A_n};\emptyset) \simeq \left( \bk, \,\, \hom_\EC(\one_\EC, A_1\otimes \cdots \otimes A_n\otimes (\oplus_{i\in \mO(\EC)} i\otimes i^\ast)^{\otimes g} ) \right).
$$
\end{cor}

\begin{rem}
Cor.\,\ref{cor:pants} provides a closed formula for GSD for topological orders on surfaces with boundaries. It generalizes the results in \cite{hw}. The result of factorization homology is completely compatible with the idea of ``$\mathrm{GSD}=\# \,\, \mbox{of confined anyons}$'' proposed in \cite{hw,lww}. 
\end{rem}

\subsection{Computing factorization homology: II} \label{sec:fh-comp2}

In this subsection, we consider another example: a stratified open cylinder $\Sigma$ with an anomaly-free coefficient system.

\medskip
In general, a stratified open cylinder can be very complicated. By applying the processes described by Example\,\ref{expl:otimes} (1),(2),(3),(4) repeatedly, however, we can always reduce the situation to a stratified open cylinder $(\Rb\times S^1;\Gamma_1)$ depicted in Fig.\,\ref{fig:cylinder-defect}. We will discuss how to compute the factorization homology of this stratified open cylinder $(\Rb\times S^1;\Gamma_1)$.

\begin{figure}[bt]
\centerline{
\raisebox{-0pt}{
  \begin{picture}(200,100)
   \put(20,20){\scalebox{3}{\includegraphics{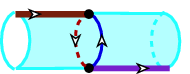}}}
   \put(20,20){
     \setlength{\unitlength}{.75pt}\put(-8,-9){
     \put(195,47)  {$\ED$}
     \put(134,47)  {$\EL$}
     \put(78,47)    {$\EK$}
     \put(55, 91)     {$ \EM $}
     \put(48, 27)     {$ \EC $}
     \put(98, 93)      {$(\EP,p)$}
     \put(98,2)    {$(\EQ,q)$}
     \put(165,2)      {$\EN$}
          }\setlength{\unitlength}{1pt}}
  \end{picture}}
\void{
&  
\raisebox{-0pt}{
  \begin{picture}(150,100)
   \put(10,20){\scalebox{3}{\includegraphics{pic-high-g-2.eps}}}
   \put(10,20){
     \setlength{\unitlength}{.75pt}\put(-8,-9){
     \put(90, 90)      {$(\EP,p)$}
     \put(147,87)     {$\ED$}
     \put(90,0)    {$(\EQ,q)$}
     \put(150, 62)     {$ Z(\ED) $}
     \put(107,43)    {$\EK^\rev\boxtimes \EL$}
     \put(42, 88)     {$ \EM $}
     \put(20,50)      {$Z(\EC)$}
     \put(148,0)      {$\EN$}
     \put(42,0)        {$\EC$}
          }\setlength{\unitlength}{1pt}}
  \end{picture}}  \\
  (a) & (b)
\end{tabular}}
}
\caption{
This figure shows a generic case of a stratified open cylinder $(\Rb\times S^1;\Gamma_1)$ after we have applied the processes described by Example\,\ref{expl:otimes} (1),(2),(3),(4). 
}
\label{fig:cylinder-defect}
\end{figure}

In this case, an anomaly-free coefficient system provides target labels to all cells.
In particular, the 2-cell labels $\EC$ and $\ED$ are UMTC's; the 1-cell labels $\EK,\EL,\EM,\EN$ are closed multi-fusion bimodules;
 the 0-cell labels $(\EP,p)$ and $(\EQ,q)$, where $\EP$ and $\EQ$ are uniquely determined by other labels up to equivalence by the anomaly-free condition. There is a unique (up to equivalence) left $\EC$-module $\EM_1$ and a right $\ED$-module $\EN_1$ such that $\EM\simeq \fun_\EC(\EM_1,\EM_1)$ and $\EN\simeq \fun_{\ED^\rev}(\EN_1,\EN_1)$ as UMFC's. Applying Rem.\,\ref{rem:red-genus} and Thm.\,\ref{thm:higher-g-defects}, we see that $$
\int_{(\Rb\times S^1;\Gamma_1)} (\EC,\ED; \EK,\EL,\EM,\EN;(\EP,p),(\EQ,q)) \simeq (\EM_1\boxtimes \EN_1^\op, \,\, u),
$$ where $u$ is determined by $\EK,\EL,(\EP,p),(\EQ,q)$.

An interesting special case is when $\EM=\EC$, $\EN=\ED$, $\EK=\EL=\EP=\EQ$ and $p=q=\one_\EK$. This case is equivalent to a cylinder $(\Rb\times S^1; S^1)$ with two 2-cells labeled by $\EC$ and $\ED$, a unique 1-cell given by an non-contractible loop $S^1$ labeled by $\EK$ and no 0-cell. The $\EC$-$\ED$-bimodule structure on $\EK$ induces two monoidal functors $\EC \xrightarrow{L} \EK \xleftarrow{R} \ED$.
\begin{prop} \label{prop:K}
We have
\be \label{eq:loop-K}
\int_{(\Rb\times S^1; S^1)} (\EC,\ED; \EK;\emptyset) \simeq \left(\EC\boxtimes \ED, \,\, \oplus_{i\in \mO(\ED)}L^\vee(R(i))\boxtimes i^\ast \right) \simeq \left(\EC\boxtimes \ED, \,\,(L\boxtimes R)^\vee(\one_\EK) \right).
\ee
\void{
If $A$ is an algebra in $\overline{\EC}\boxtimes \ED$ such that there is an equivalence $f: \EK \xrightarrow{\simeq} (\overline{\EC}\boxtimes \ED)_A$ of UMFC's and a natural isomorphism $f\circ (L\boxtimes R)\simeq-\otimes A: \overline{\EC}\boxtimes \ED \to (\overline{\EC}\boxtimes \ED)_A$, then we have
\be
\int_{(\Rb\times S^1; S^1)} (\EC, \ED; \EK;\emptyset) \simeq (\EC\boxtimes \ED, \, A).
\ee
}
\end{prop}
\pf
Applying Lem.\,\ref{lem:M-cylinder} and Rem.\,\ref{rem:red-genus} to the cylinder labeled by $\ED$, then applying Prop.\,\ref{prop:2disk}, we obtain
$$
\int_{(\Rb\times S^1; S^1)} (\EC,\ED; \EK;\emptyset) \simeq \left(\EC\boxtimes \ED^\op, \,\, \oplus_{i\in \mO(\ED)}L^\vee(R(i))\boxtimes i \right).
$$
Using the equivalence $\delta: \ED^\op \to \ED$, we obtain the first equivalence of Eq.\,(\ref{eq:loop-K}) immediately. 
The second equivalence follows from the following identities:
\begin{align}
&\hom_{\EC\boxtimes \ED}( j\boxtimes k,\,\,\,\, \oplus_{i\in \mO(\ED)}L^\vee(R(i))\boxtimes i^\ast) \simeq \hom_\EC(j,L^\vee(R(i))) \otimes \hom_\ED(k,i^\ast)
\simeq \hom_\EC(L(j), R(k^\ast))  \nn
&\hspace{1cm} \simeq \hom_\EK(L(j)\otimes R(k), \one_\EK) \simeq \hom_{\EC\boxtimes \ED} \left( j\boxtimes k, (L\boxtimes R)^\vee(\one_\EK) \right),
\nonumber
\end{align}
for all $j\in \mO(\EC)$, $k\in \mO(\ED)$.
\epf

\begin{rem}
The formula Eq.\,(\ref{eq:loop-K}) suggests a very useful tool to compute factorization homology of a stratified surface (open or closed), in which there is a cylinder shaped region with a single non-contractible loop labeled by a closed fusion $\EC$-$\ED$-bimodule $\EK$. One can always replace this region by two disjoint open 2-disks with 2-cells labeled by $\EC$ and $\ED$ and two 0-cells (on different 2-disks) simultaneously labeled by $\oplus_{i\in \mO(\ED)}L^\vee(R(i))\boxtimes i^\ast$. This replacement reduces the genus of the surface by one.
\end{rem}

\begin{cor}
If a stratified sphere $(S^2; S^1)$ with two 2-cells labeled by UMTC's $\EC$ and $\ED$ and the unique 1-cell (a loop $S^1$) labeled by a closed multi-fusion $\EC$-$\ED$-bimodule $\EK$, then we have.
\be \label{eq:cdm-A}
\int_{(S^2; S^1)} (\EC, \ED; \EK;\emptyset) \simeq (\bk, \hom_{\EC\boxtimes \ED}(\one_\EC \boxtimes \one_\ED, \,\, A)),
\ee
where $A=(L\boxtimes R)^\vee(\one_\EK)$ is a commutative separable algebra in $\overline{\EC}\boxtimes\ED$.
\end{cor}

\begin{rem}
Eq.\,(\ref{eq:cdm-A}) exactly gives the GSD of a stratified sphere $(S^2; S^1)$ with two 2d phases given by UMTC's $\EC$ and $\ED$, respectively, and bounded by an anomaly-free gapped domain wall $\EK$ \cite{lww}. The wall $\EK$ is called {\it stable} if and only if the GSD on the stratified sphere $(S^2; S^1)$ is trivial \cite{lww,kong-wen-zheng}. In other words, the wall $\EK$ is stable iff $\hom_{\EC\boxtimes \ED}(\one_\EC \boxtimes \one_\ED, A)\simeq\Cb$, or equivalently, iff $\EK$ is a unitary fusion category. If $\EK$ is not fusion, $\dim \hom_{\EC\boxtimes\ED}(\one_\EC \boxtimes \one_\ED,A) = \dim \hom_\EK(\one_\EK,\one_\EK) > 1$. In this case, the wall $\EK$ is not stable under the perturbation of local operators. These perturbations can let $\EK$ flow to a stable wall described by a unitary fusion category. An example of such situation is given in Example\,\ref{expl:toric-2}, where $\EK$ must flow to the trivial 1d phase $\bk$. In general, unitary fusion categories in the same Morita class are not unique. Therefore, it becomes an interesting problem to work out to which unitary fusion category $\EK$ can flow.
\end{rem}

Let $\EC,\ED,\EE$ be UMTC's, $\EM$ a closed fusion $\EC$-$\ED$-bimodule and $\EN$ a closed fusion $\ED$-$\EE$-bimodule.
We have canonical monoidal functors $\EC\xrightarrow{L_\EM} \EM \xleftarrow{R_\EM} \ED$, $\ED\xrightarrow{L_\EN} \EN \xleftarrow{R_\EN} \EE$, and $\EC \xrightarrow{L} (\EM\boxtimes_\ED\EN) \xleftarrow{R} \EE$.
It is very easy to determine $(L\boxtimes R)^\vee(\one_{\EM\boxtimes_\ED\EN})$ as an object in $\overline{\EC}\boxtimes \EE$ via factorization homology as follows.
\void{
by replacing a cylinder with a unique 2-cell labeled by $\ED$ by two open disks with 2-cell labeled by $\ED$ and $\ED$ and two 0-cell simultaneously labeled by $\oplus_{i\in \mO(\ED)}i\boxtimes i^\ast$. More precisely, we have
$$
(L\boxtimes R)^\vee(\one_{\EM\boxtimes_\ED\EN}) \simeq \oplus_{i\in \mO(\ED)} L_\EM^\vee(R_\EM(i)) \boxtimes R_\EN^\vee(L_\EN(i^\ast)).
$$
One can reexpress $(L\boxtimes R)^\vee(\one_{\EM\boxtimes_\ED\EN})$ in terms of more concrete data. 
}
We define the $W$-matrix $W^\EM$ for $\EM$ as:
\be \label{eq:W-matrix}
(L_\EM\boxtimes R_\EM)^\vee(\one_\EM) = \bigoplus_{i\in \mO(\EC), j\in \mO(\ED)} W_{ij}^\EM \,\,\, i\boxtimes j^\ast
\ee
and define $W^\EN$ for $\EN$ similarly. Then we have
$$
(L\boxtimes R)^\vee(\one_{\EM\boxtimes_\ED\EN}) \simeq \bigoplus_{i\in \mO(\EC), j\in \mO(\EE)} (W^\EM W^\EN)_{ij}\,\, i\boxtimes j^\ast.
$$

As a consequence, we obtain the following result.

\begin{cor}
Let $(\Rb\times S^1; S^1\cup\cdots\cup S^1)$ be a stratified open cylinder with 2-cells labeled by UMTC's $\EC_0,\cdots,\EC_n$ and 1-cells ($n$ disconnected non-contractible loops) labeled by closed multi-fusion $\EC_{i-1}$-$\EC_i$-bimodules $\EM_i$, $i=1,\cdots, n$ and no 0-cell. Let the matrix $W^{\EM_i}$, $i=1,\cdots,n$ be defined by Eq.\,(\ref{eq:W-matrix}). We have
$$
\int_{(\Rb\times S^1; S^1\cup\cdots\cup S^1)} (\EC_0, \cdots, \EC_n; \EM_1,\cdots, \EM_n;\emptyset) \simeq \left(\EC_0\boxtimes \EC_n, \,\,  \bigoplus_{i\in \mO(\EC_0), j\in \mO(\EC_n)} (W^{\EM_1}W^{\EM_2} \cdots W^{\EM_n})_{ij} \,\, i\boxtimes j^\ast \right).
$$
If we glue two ends of the cylinder to get a stratified torus $(S^1\times S^1; S^1\cup \cdots \cup S^1)$, we have
$$
\int_{(S^1\times S^1; S^1\cup\cdots\cup S^1)} (\EC_0, \cdots, \EC_{n-1}; \EM_1,\cdots, \EM_n;\emptyset) \simeq (\bk, \,\,  u)
$$
where $\dim u = \mathrm{Tr}(W^{\EM_1}W^{\EM_2} \cdots W^{\EM_n})$.
\end{cor}

\begin{rem}
The formula $\mathrm{Tr}(W^{\EM_1}W^{\EM_2} \cdots W^{\EM_n})$ first appeared in the work of Lan-Wang-Wen in \cite{lww} as a formula for the GSD for the same stratified torus with the same target labels $\EC_i$ and $\EM_j$. It is tautological to check that other formulas for GSD appeared in \cite{lww} also match precisely with the result of factorization homology. This is because the calculation of GSD in \cite{lww} makes use of the way that anyons tunnel through the wall on the stratified cylinder from Prop.\,\ref{prop:K}, and this tunneling process matches precisely with the result of factorization homology given in Prop.\,\ref{prop:K}.
\end{rem}

\end{document}